\documentclass{amsart}

\usepackage{amsmath, amssymb, amsthm, bbm}

\newtheorem{theorem}{Theorem}
\newtheorem{lemma}[theorem]{Lemma}
\newtheorem{proposition}[theorem]{Proposition}
\newtheorem{corollary}[theorem]{Corollary}
\newtheorem{claim}[theorem]{Claim}

\newtheorem{fact}[theorem]{Fact}
\newtheorem{question}[theorem]{Question}

\newenvironment{definition}[1][Definition]{\begin{trivlist}
\item[\hskip \labelsep {\bfseries #1}]}{\end{trivlist}}

\newenvironment{remark}[1][Remark]{\begin{trivlist}
\item[\hskip \labelsep {\bfseries #1}]}{\end{trivlist}}

\newcommand{\cf}{\mathrm{cf}}
\newcommand{\dom}{\mathrm{dom}}
\newcommand{\bb}{\mathbb}
\newcommand{\bbm}{\mathbbm}

\begin{document}

\title{Bounded Stationary Reflection II}
\author{Chris Lambie-Hanson}
\address{Einstein Institute of Mathematics, Hebrew University of Jerusalem \\ Jerusalem, 91904, Israel}
\email{clambiehanson@math.huji.ac.il}
\thanks{This research was undertaken while the author was a Lady Davis Postdoctoral Fellow. The author would like to thank the Lady Davis Fellowship Trust and the Hebrew University of Jerusalem. The author would also like to thank Menachem Magidor for many helpful discussions.}

\begin{abstract}
	Bounded stationary reflection at a cardinal $\lambda$ is the assertion that every stationary subset of $\lambda$ reflects but there is a stationary subset of $\lambda$ that does not reflect at arbitrarily high cofinalities. We produce a variety of models in which bounded stationary reflection holds. These include models in which bounded stationary reflection holds at the successor of every singular cardinal $\mu > \aleph_\omega$ and models in which bounded stationary reflection holds at $\mu^+$ but the approachability property fails at $\mu$.
\end{abstract}

\maketitle

\section{Introduction}

The reflection of stationary sets is a topic of fundamental interest in the study of combinatorial set theory, large cardinals, and inner model theory and provides a useful tool for the investigation of the tension between compactness and incompactness phenomena. In this paper, we extend results, inspired by a question of Eisworth, of Cummings and the author \cite{reflection}. We start by reviewing the relevant definitions and providing an outline of the structure of the paper.

Our notation is for the most part standard. The primary reference for all undefined notions and notations is \cite{jech}. If $\kappa < \lambda$ are infinite cardinals, with $\kappa$ regular, then $S^\lambda_\kappa = \{\alpha < \lambda \mid \cf(\alpha) = \kappa \}$. Expressions such as $S^\lambda_{>\kappa}$ or $S^\lambda_{\geq \kappa}$ are defined in the obvious way. If $X$ is a set of ordinals, then $\mathrm{nacc}(X)$ (the set of non-accumulation points of $X$) is the set $\{\alpha \in X \mid \sup(X \cap \alpha) < \alpha\}$, $X'$ is the set of limit points of $X$ (i.e. $X \setminus \mathrm{nacc}(X)$), and $\mathrm{otp}(X)$ is the order type of $X$. If $\langle \bb{P}_\xi, \dot{\bb{Q}}_\zeta \mid \xi \leq \gamma, \zeta < \gamma \rangle$ is a forcing iteration with supports of size $\mu$ for some cardinal $\mu$, we will frequently write $\Vdash_\xi$ instead of $\Vdash_{\bb{P}_\xi}$. Conditions of $\bb{P}_\gamma$ are thought of as functions $p$ such that $\dom(p) \in [\gamma]^{\leq \mu}$ and, for all $\zeta \in \dom(p)$, $\Vdash_{\zeta}``p(\zeta) \in \dot{\bb{Q}}_\zeta."$ For $\zeta < \xi \leq \gamma$, we let $\dot{\bb{P}}_{\zeta, \xi}$ be a $\bb{P}_\zeta$-name such that $\bb{P}_\xi \cong \bb{P}_\zeta * \dot{\bb{P}}_{\zeta, \xi}$.

\begin{definition}
	Let $\lambda > \omega_1$ be a regular cardinal.
	\begin{enumerate}
		\item{If $S \subseteq \lambda$ is a stationary set and $\alpha \in S^{\lambda}_{>\omega}$, then \emph{$S$ reflects at $\alpha$} if $S \cap \alpha$ is stationary in $\alpha$. $S$ \emph{reflects} if there is $\alpha \in S^\lambda_{>\omega}$ such that $S$ reflects at $\alpha$.}
		\item{If $\mu$ is a singular cardinal and $\lambda = \mu^+$, then $\mathrm{Refl}(\lambda)$ holds if every stationary subset of $\lambda$ reflects.}
		\item{If $\mu$ is a singular cardinal, $\lambda = \mu^+$, and $S \subseteq \lambda$ is stationary, then $S$ \emph{reflects at arbitrarily high cofinalities} if, for all $\kappa < \mu$, there is $\alpha \in S^\lambda_{\geq \kappa}$ such that $S$ reflects at $\alpha$.}
		\item{If $\mu$ is a singular cardinal and $\lambda = \mu^+$, then $\mathrm{bRefl}(\lambda)$ (bounded stationary reflection at $\lambda$) holds if $\mathrm{Refl}(\lambda)$ holds but there is a stationary $T \subseteq \lambda$ that does not reflect at arbitrarily high cofinalities.}
	\end{enumerate}
\end{definition}

Eisworth \cite{ToddEmail} asked whether $\mathrm{bRefl}(\lambda)$ is consistent when $\lambda$ is the successor of a singular cardinal. $\mathrm{bRefl}(\aleph_{\omega+1})$ is easily seen to be inconsistent, but Cummings and the author showed in \cite{reflection} that, for other values of $\lambda$, $\mathrm{bRefl}(\lambda)$ is consistent modulo large cardinal assumptions. In particular, the following theorem was proven.

\begin{theorem}
	Suppose there is a proper class of supercompact cardinals. Then there is a class forcing extension in which, for every singular cardinal $\mu > \aleph_\omega$ such that $\mu$ is not a cardinal fixed point, $\mathrm{bRefl}(\mu^+)$ holds.
\end{theorem}

This left open the question of whether it is consistent that $\mathrm{bRefl}(\mu^+)$ holds for every singular cardinal $\mu > \aleph_\omega$. In this paper, we answer this question affirmatively and prove a number of variations on the main result from \cite{reflection}.

In Section \ref{forcingsect}, we briefly discuss the notion of approachability before defining some of the forcing posets to be used throughout the paper and introducing their basic properties. In Section \ref{destroySection}, we prove a general lemma about iteratively destroying stationary sets. In Section \ref{densesect}, we prove a dense version of the result from \cite{reflection} by producing a model in which $\mathrm{Refl}(\aleph_{\omega \cdot 2 + 1})$ holds and, for every stationary $S \subseteq S^{\aleph_{\omega \cdot 2 + 1}}_{<\aleph_\omega}$, there is a stationary $T \subseteq S$ that does not reflect at arbitrarily high cofinalities. In Section \ref{globalsect}, we prove a global version of the result from \cite{reflection} by producing a model in which, for every singular cardinal $\mu > \aleph_\omega$, $\mathrm{bRefl}(\mu^+)$ holds.

The proofs of the results in Sections \ref{densesect} and \ref{globalsect} and in \cite{reflection} rely heavily on the approachability property holding in the final model. The relationship between approachability and stationary reflection is complicated and interesting, and in the last two sections of this paper we investigate the extent to which we can get bounded stationary reflection together with the failure of approachability. In Section \ref{approachSect1}, we produce a model with a singular cardinal $\mu$ such that $AP_\mu$ fails and $\mathrm{bRefl}(\mu^+)$ holds. In this model $\mu$ is a limit of cardinals which are supercompact in an outer model. In Section \ref{approachSect2}, we show that this result can be attained with $\mu = \aleph_{\omega^2 \cdot 2}$.

\section{Approachability and forcing preliminaries} \label{forcingsect}

\begin{definition}
	Let $\lambda$ be a regular, uncountable cardinal.
	\begin{enumerate}
		\item{Let $\vec{a} = \langle a_\alpha \mid \alpha < \lambda \rangle$ be a sequence of bounded subsets of $\lambda$. If $\gamma < \lambda$, $\gamma$ is \emph{approachable with respect to $\vec{a}$} if there is an unbounded $A \subseteq \gamma$ such that $\mathrm{otp}(A) = \cf(\gamma)$ and, for every $\beta < \gamma$, there is $\alpha < \gamma$ such that $A \cap \beta = a_\gamma$.}
		\item{If $B \subseteq \lambda$, then $B \in I[\lambda]$ if there is a club $C \subseteq \lambda$ and a sequence $\vec{a} = \langle a_\alpha \mid \alpha < \lambda \rangle$ of bounded subsets of $\lambda$ such that every element of $B \cap C$ is approachable with respect to $\vec{a}$.}
		\item{If $\mu$ is a singular cardinal and $\lambda = \mu^+$, then $AP_\mu$ holds if $\lambda \in I[\lambda]$.}
	\end{enumerate}
\end{definition}

A wealth of information about approachability, including proofs of the statements in the following remark, can be found in \cite{eisworth}.

\begin{remark}
	Let $\lambda$ be a regular, uncountable cardinal.
	\begin{enumerate}
		\item{$I[\lambda]$ is a normal, $\lambda$-complete ideal extending the non-stationary ideal on $\lambda$.}
		\item{Suppose $\lambda^{<\lambda} = \lambda$ and $\vec{a} = \langle a_\alpha \mid \alpha < \lambda \rangle$ is a fixed enumeration of all bounded subsets of $\lambda$. If $B \subseteq \lambda$, then $B \in I[\lambda]$ iff there is a club $C \subseteq \lambda$ such that every element of $B \cap C$ is approachable with respect to $\vec{a}$.}
	\end{enumerate}
\end{remark}

\begin{definition}
	Let $\theta$ be a regular cardinal, and let $\vartriangleleft$ be a well-ordering of $H(\theta)$. An \emph{internally approachable chain of substructures of $H(\theta)$} is an increasing, continuous sequence $\langle M_i \mid i < \eta \rangle$ such that, for all $i < \eta$:
	\begin{itemize}
		\item{$M_i \prec (H(\theta), \in, \vartriangleleft)$.}
		\item{$\langle M_k \mid k \leq i \rangle \in M_{i+1}$.}
	\end{itemize}
\end{definition}

The notion of approachability is intimately connected with internally approachable chains. The following result is proven in \cite{fm}.

\begin{lemma} \label{ialemma}
	Let $\lambda < \theta$ be regular cardinals, let $x \in H(\theta)$, and let $\vartriangleleft$ be a well-ordering of $H(\theta)$. Suppose $S \in I[\lambda]$. Then there is a club $C \subseteq \lambda$ such that, for every $\gamma \in C \cap S$, letting $\kappa = \cf(\gamma)$, there is an internally approachable chain $\langle M_i \mid i < \kappa \rangle$ of elementary substructures of $(H(\theta), \in, \vartriangleleft)$ such that:
	\begin{enumerate}
		\item{For all $i < \kappa$, $|M_i| < \kappa$.}
		\item{$x \in M_0$.}
		\item{If $M = \bigcup_{i < \kappa} M_i$, then $\gamma = \sup(M \cap \lambda)$.}
	\end{enumerate}
\end{lemma}

Before we introduce specific forcing posets, we recall the notion of strategic closure.

\begin{definition}
	Let $\mathbb{P}$ be a partial order and let $\beta$ be an ordinal.
	\begin{enumerate}
		\item {The two-player game $G_\beta(\mathbb{P})$ is defined as follows: Players I and II alternately play entries in $\langle p_\alpha \mid \alpha < \beta \rangle$, a decreasing sequence of conditions in $\mathbb{P}$ with $p_0 = \bbm{1}_{\mathbb{P}}$. Player I plays at odd stages, and Player II plays at even stages (including all limit stages). If there is an even stage $\alpha < \beta$ at which Player II cannot play, then Player I wins. Otherwise, Player II wins.}
		\item{$G_\beta^*(\mathbb{P})$ is defined just as $G_\beta(\mathbb{P})$ except Player II no longer plays at limit stages. Instead, if $\alpha < \beta$ is a limit ordinal, then $p_\alpha = \inf(\{p_\gamma \mid \gamma < \alpha \})$ if such a condition exists. If, for some limit $\alpha < \beta$, such a condition does not exist, then Player I wins. Otherwise, Player II wins.}
		\item{$\mathbb{P}$ is {\em $\beta$-strategically closed} if Player II has a winning strategy for the game $G_\beta(\mathbb{P})$. $\mathbb{P}$ is {\em strongly $\beta$-strategically closed} if Player II has a winning strategy for the game $G_\beta^*(\mathbb{P})$. The notions of \emph{{$<\beta$-strategically} closed} and {\em strongly $<\beta$-strategically closed} are defined in the obvious way.}
	\end{enumerate}
\end{definition}

We now define a number of forcing posets to be used throughout the paper and state some of their basic properties. We first introduce a poset used to force $AP_\mu$. Suppose $\mu$ is a singular cardinal, $\lambda = \mu^+$, and $\lambda^{<\lambda} = \lambda$. Let $\vec{a} = \langle a_\alpha \mid \alpha < \lambda \rangle$ be an enumeration of the bounded subsets of $\lambda$. Let $S$ be the set of ordinals that are approachable with respect to $\vec{a}$. The poset $\bb{A}_{\vec{a}}$ consists of closed, bounded subsets of $S$ and is ordered by end-extension, i.e. $c \leq d$ if $c \cap (\max(d)+1) = d$. The following is due to Shelah.

\begin{proposition}
	$\bb{A}_{\vec{a}}$ is strongly $(<\lambda)$-strategically closed.
\end{proposition}

Next, we define a poset to add a stationary set that only reflects at points of small cofinality. Let $\kappa < \mu < \lambda$ be infinite, regular cardinals. Conditions in $\bb{S}^\lambda_{\kappa, \mu}$ are functions $s \in {^{\gamma_s + 1}2}$ such that:
\begin{enumerate}
	\item{$\gamma_s < \lambda$.}
	\item{$\{\alpha \leq \gamma_s \mid s(\alpha) = 1\} \subseteq S^\lambda_\kappa$.}
	\item{$\{\alpha \leq \gamma_s \mid s(\alpha) = 1\} \cap \beta$ is nonstationary in $\beta$ for all $\beta \in S^\lambda_{\geq \mu}$.}
\end{enumerate}
Conditions are ordered by reverse inclusion. We will sometimes abuse notation and identify $s$ with $\{\alpha \leq \gamma_s \mid s(\alpha) = 1\}$ in statements such as ``$s$ does not reflect at any $\beta \in S^\lambda_{\geq \mu}$." Note, however, that recovering $s$ from $\{\alpha \leq \gamma_s \mid s(\alpha) = 1\}$ requires the parameter $\gamma_s$. 

Proofs of the following facts can be found in \cite{reflection}.

\begin{lemma}
	Let $\bb{S} = \bb{S}^\lambda_{\kappa, \mu}$.
	\begin{enumerate}
		\item{$\bb{S}$ is $\mu$-directed closed.}
		\item{$\bb{S}$ is $<\lambda$-strategically closed.}
		\item{Let $G$ be $\bb{S}$-generic over $V$, and let $S = \{\alpha < \lambda \mid$ for some $s\in G$, $s(\alpha) = 1\}$. Then, in $V[G]$, $S$ is a stationary subset of $S^\lambda_\kappa$ that does not reflect at any ordinal in $S^\lambda_{\geq \mu}$.}
	\end{enumerate}
\end{lemma}

For some constructions we will need a variant of $\bb{S}^\lambda_{\kappa, \mu}$. Let $\mu < \lambda$ be infinite, regular cardinals, and let $T\subseteq S^\lambda_{<\mu}$ be stationary. $\bb{S}_{T, \mu}$ is defined exactly as $\bb{S}^\lambda_{\kappa, \mu}$ except that, for $s \in \bb{S}_{T, \mu},$ instead of clause (2), we require that, if $s(\alpha) = 1$, then $\alpha \in T$. Note that $\bb{S}^\lambda_{\kappa, \mu} = \bb{S}_{S^\lambda_\kappa, \mu}$. The purpose of $\bb{S}_{T, \mu}$ is to add a subset of $T$ that does not reflect at any ordinals in $S^\lambda_{\geq \mu}$. The same arguments used for $\bb{S}^\lambda_{\kappa, \mu}$ show that $\bb{S}_{T, \mu}$ is $\mu$-directed closed and $<\lambda$-strategically closed. If $S$ (with canonical name $\dot{S}$) is the subset of $T$ added by $\bb{S}_{T, \mu}$, it is clear that $S$ does not reflect at any ordinals in $S^\lambda_{\geq \mu}$. With an additional assumption, we can ensure as well that $S$ is stationary.

\begin{lemma} \label{genstat}
	Suppose that $T \in I[\lambda]$. Then $\Vdash_{\bb{S}_{T, \mu}} ``\dot{S}$ is stationary."
\end{lemma}

\begin{proof}
	Let $\bb{S} = \bb{S}_{T, \mu}$. Let $\dot{C}$ be an $\bb{S}$-name for a club in $\lambda$, and let $s \in \bb{S}$. We will find $s^* \leq s$ and $\beta < \lambda$ such that $s^* \Vdash ``\beta \in \dot{S} \cap \dot{C}."$

	Let $\theta$ be a sufficiently large regular cardinal. Since $T \in I[\lambda]$, we can find $M \prec H(\theta)$ such that:
	\begin{itemize}
		\item{$\beta:= \sup(M \cap \lambda) \in T$.}
		\item{$|M| = \kappa$, where $\kappa = \cf(\beta)$.}
		\item{$M$ is the union of an internally approachable chain $\langle M_i \mid i < \kappa \rangle$, where $|M_i| < \kappa$ for all $i < \kappa$.}
		\item{All relevant information (including $\bb{S}$, $\dot{C}$, and $s$) is in $M_0$.}
	\end{itemize}

	We now construct a decreasing sequence of conditions $\langle s_i \mid i < \kappa \rangle$ such that:
	\begin{itemize}
		\item{$s_0 \leq s$.}
		\item{For all $i < \kappa$, $s_i \in M_{i + 1}$.}
		\item{For all $i < \kappa$, there is $\alpha_i \geq \sup(M_i \cap \lambda)$ such that $s_i \Vdash ``\alpha_i \in \dot{C}."$}
	\end{itemize}
	The construction is a straightforward recursion using the $\mu$-closure of $\bb{S}$ and maintaining the additional requirement, made possible by the internal approachability of $\langle M_i \mid i < \kappa \rangle$, that, for every $i < \kappa$, $\langle s_j \mid j < i \rangle \in M_{i+1}$. Now let $s^* = \{(\beta, 1)\} \cup \bigcup_{i < \kappa} s_i$. $s^*$ is easily seen to be a member of $\bb{S}$, and $s^* \Vdash ``\beta \in \dot{S} \cap \dot{C}."$
\end{proof}

We now introduce a well-known poset used to destroy the stationarity of subsets of $\lambda$, where $\lambda$ is an uncountable, regular cardinal. Let $S$ be a subset of $\lambda$, and let $CU(S)$ consist of closed, bounded $t \subset \lambda$ such that $t \cap S = \emptyset$ and, if $\alpha \in \mathrm{nacc}(t)$, then $\cf(\alpha) = \omega$. (This last condition is not strictly necessary, but it will make certain technical points simpler.) We denote $\max(t)$ by $\gamma_t$. If $G$ is $CU(S)$-generic over $V$, then $S$ is no longer stationary in $V[G]$. In general, forcing with $CU(S)$ can collapse cardinals. However, if $S$ was just added by $\bb{S}^\lambda_{\kappa, \mu}$ or $\bb{S}_{T, \mu}$, then the forcing is quite nice. 

\begin{lemma}
	Let $\bb{S} = \bb{S}^\lambda_{\kappa, \mu}$ (or $\bb{S}_{T, \mu})$, and let $\dot{S}$ be a name for the stationary subset of $\lambda$ added by $\bb{S}$. Then $\bb{S} * CU(\dot{S})$ has a $\lambda$-closed dense subset.
\end{lemma}

The reader is again referred to \cite{reflection} for the proof. For completeness, we mention that the $\lambda$-closed dense subset is the set of $(s, \dot{t})$ such that there is $t^* \in V$ such that $s \Vdash ``\dot{t} = t^*"$ and $\gamma_{s} = \gamma_{t^*}$.

We will need the following well-known facts:

\begin{fact} \cite{magidor} \label{lift}
	Let $\kappa$ be a regular cardinal, and let $\kappa<\lambda<\mu$. Suppose that, in $V^{\mathrm{Coll}(\kappa, <\lambda)}$, $\mathbb{P}$ is a separative, strongly $\kappa$-strategically closed partial order and $|\mathbb{P}|<\mu$. Let $i$ be the natural complete embedding of $\mathrm{Coll}(\kappa, <\lambda)$ into $\mathrm{Coll}(\kappa, <\mu)$ (namely, the identity embedding). Then $i$ can be extended to a complete embedding $j$ of $\mathrm{Coll}(\kappa, <\lambda)*\mathbb{P}$ into $\mathrm{Coll}(\kappa, <\mu)$ so that the quotient forcing $\mathrm{Coll}(\kappa, <\mu)/j[\mathrm{Coll}(\kappa, <\lambda)*\mathbb{P}]$ is $\kappa$-closed.
\end{fact}

\begin{fact} \label{apstat}
	Let $\mu$ and $\kappa$ be cardinals. Suppose that $AP_\kappa$ holds, $S$ is a stationary subset of $S^{\kappa^+}_{<\mu}$, and $\mathbb{P}$ is a $\mu$-closed forcing poset. Then $S$ remains stationary in $V^{\mathbb{P}}$.
\end{fact}

\section{Destroying stationary sets} \label{destroySection}

At many points in this paper, we will want to use a forcing iteration to destroy the stationary of many sets. We will also want to ensure that we do not collapse any cardinals in the process. With this in mind, we prove here a general lemma which we will use, either directly or via a modification, throughout the remainder of the paper.

Let $\lambda$ be a regular cardinal, and assume that $\lambda^{<\lambda} = \lambda$ and $2^\lambda = \lambda^+$. Let $\bb{S}$ be a forcing poset, and let $\dot{\bb{T}}$ be an $\bb{S}$-name for a forcing poset such that $\bb{S} * \dot{\bb{T}}$ has a dense $\lambda$-closed subset (note that, in particular, this implies that $\bb{S}$ is $\lambda$-distributive and $\Vdash_{\bb{S}}``\dot{\bb{T}}$ is $\lambda$-distributive").

In $V^{\bb{S}}$, let $\langle \bb{P}_\xi, \dot{\bb{Q}}_\zeta \mid \xi \leq \lambda^+, \zeta < \lambda^+ \rangle$ be a forcing iteration with supports of size $<\lambda$ such that, for every $\zeta < \lambda^+$, there is a $\bb{P}_\zeta$-name $\dot{T}_\zeta$ for a subset of $\lambda$ such that $\Vdash_{\bb{P}_\zeta * \bb{T}}``\dot{T}_\zeta$ is non-stationary" and $\Vdash_{\bb{P}_\zeta}``\dot{\bb{Q}}_\zeta = CU(\dot{T}_\zeta)."$ By an easy $\Delta$-system argument, $\bb{P}$ has the $\lambda^+$-c.c. 

\begin{lemma} \label{stat_dest}
	$\bb{S} * \dot{\bb{P}} * \dot{\bb{T}}$ has a dense $\lambda$-closed subset.
\end{lemma}

\begin{proof}
	For each $\zeta < \lambda^+$, let $\dot{C}_\zeta$ be an $\bb{S} * \dot{\bb{P}}_\zeta * \dot{\bb{T}}$-name for a club in $\lambda$ disjoint from $\dot{T}_\zeta$.

	Let $\bb{U}_0$ be the dense $\lambda$-closed subset of $\bb{S} * \dot{\bb{T}}$ that exists by assumption. Let $\bb{U}$ be the set of $(s, \dot{p}, \dot{t}) \in \bb{S} * \dot{\bb{P}} * \dot{\bb{T}}$ such that:

	\begin{itemize}
		\item{$(s, \dot{t}) \in \bb{U}_0$.}
		\item{$s$ decides the value of $\dom(\dot{p})$.}
		\item{For every $\zeta \in \dom(\dot{p})$, $(s, \dot{p} \restriction \zeta, \dot{t}) \Vdash_{\bb{S}*\dot{\bb{P}}_\zeta * \dot{\bb{T}}}``\max(\dot{p}(\zeta)) \in \dot{C}_\zeta"$.}
	\end{itemize}
	For $\zeta < \lambda^+$, let $\bb{U}_\zeta$ be the set of $(s, \dot{p}, \dot{t}) \in \bb{U}$ such that $\dom(\dot{p}) \subseteq \zeta$. We will prove by induction on $\zeta$ that $\bb{U}_\zeta$ is a dense, $\lambda$-closed subset of $\bb{S} * \dot{\bb{P}}_\zeta * \dot{\bb{T}}$ for all $\zeta < \lambda^+$, which will imply that $\bb{U}$ is a dense, $\lambda$-closed subset of $\bb{S} * \dot{\bb{P}} * \dot{\bb{T}}$.

	Thus, fix $\zeta < \lambda^+$, and assume we have proven that $\bb{U}_\xi$ is a $\lambda$-closed dense subset of $\bb{S}*\dot{\bb{P}}_\xi * \dot{\bb{T}}$ for all $\xi < \zeta$. We first show that $\bb{U}_\zeta$ is $\lambda$-closed. Let $\beta < \lambda$, and let $\langle (s_\alpha, \dot{p}_\alpha, \dot{t}_\alpha) \mid \alpha < \beta \rangle$ be a decreasing sequence of conditions from $\bb{U}_\zeta$. Let $X = \bigcup_{\alpha < \beta} \dom(\dot{p}_\alpha)$, and note that $X \in [\zeta]^{<\lambda}$. We will define a lower bound $(s^*, \dot{p}^*, \dot{t}^*) \in \bb{U}_\zeta$ such that $\dom(\dot{p}^*) = X$.

	First, let $(s^*, \dot{t}^*) \in \bb{U}_0$ be a lower bound for $\langle (s_\alpha ,\dot{t}_\alpha) \mid \alpha < \beta \rangle$. We define $\dot{p}^*$ by defining $\dot{p}^* \restriction \xi$ by induction on $\xi \leq \zeta$. So, suppose $\xi < \zeta$ and $\dot{p}^* \restriction \xi$ has been defined so that $(s^*, \dot{p}^* \restriction \xi, \dot{t}^*) \in \bb{U}_\xi$. If $\xi \not\in X$, then $\dot{p}^* \restriction (\xi + 1) = \dot{p}^* \restriction \xi$. Thus, suppose $\xi \in X$. Let $\alpha_\xi$ be the least $\alpha < \beta$ such that $\xi \in \dom(\dot{p}_\alpha)$. Let $\dot{\gamma}_\xi$ be an $\bb{S} * \dot{\bb{P}}_{\xi}$-name for $\sup(\{\max(\dot{p}_\alpha(\xi)) \mid \alpha_\xi \leq \alpha < \beta \})$, and let $\dot{p}^*(\xi)$ be an $\bb{S} * \dot{\bb{P}}_\xi$-name for $\{\dot{\gamma}_\xi\} \cup \bigcup_{\alpha_\xi \leq \alpha < \beta} \dot{p}_\alpha(\xi)$. It is clear that $\dot{p}^*(\xi)$ is an $\bb{S}*\dot{\bb{P}}_\xi$-name for a closed, bounded subset of $\lambda$ and that, if $(s^*, \dot{p}^* \restriction (\xi + 1), \dot{t})$ is in fact in $\bb{U}_{\xi + 1}$, then it is a lower bound for $\langle (s_\alpha, \dot{p}_\alpha \restriction (\xi + 1), \dot{t}_\alpha) \mid \alpha < \beta \rangle$. It thus remains to check that $(s^*, \dot{p}^* \restriction \xi, \dot{t}) \Vdash_{\bb{S}*\dot{\bb{P}}_\xi * \dot{\bb{T}}}``\dot{\gamma}_\xi \in \dot{C}_\xi"$.

	Suppose this is not the case. Find $u \leq (s^*, \dot{p}^* \restriction \xi, \dot{t})$ in $\bb{U}_\xi$ such that:
	\begin{enumerate}
		\item{$u \Vdash ``\dot{\gamma}_\xi \not\in \dot{C}_\xi"$.}
		\item{For every $\alpha_\xi \leq \alpha < \beta$, there is $\gamma^\xi_\alpha < \lambda$ such that $u \Vdash ``\max(\dot{p}_\alpha(\xi)) = \gamma^\xi_\alpha"$.}
	\end{enumerate}
	This is possible due to our inductive assumption that $\bb{U}_\xi$ is a dense, $\lambda$-closed subset of $\bb{S} * \dot{\bb{P}}_\xi * \dot{\bb{T}}$. Thus, for every $\alpha_\xi \leq \alpha < \beta$, $u \Vdash``\gamma^\xi_\alpha \in \dot{C}_\xi"$. Also, $u \Vdash ``\dot{\gamma}_\xi = \sup(\{\gamma^\xi_\alpha \mid \alpha_\xi \leq \alpha < \beta \})$ and $\dot{C}_\xi$ is club in $\lambda"$. Thus, $u \Vdash ``\dot{\gamma}_\xi \in \dot{C}_\xi"$, which is a contradiction. We have thus completed the construction of $(s^*, \dot{p}^*, \dot{t}^*)$ and the proof that $\bb{U}_\zeta$ is $\lambda$-closed.

	We now prove that $\bb{U}_\zeta$ is dense in $\bb{S}*\dot{\bb{P}}_\zeta * \dot{\bb{T}}$. Let $(s, \dot{p}, \dot{t}) \in \bb{S}*\dot{\bb{P}}_\zeta * \dot{\bb{T}}$. Since $\bb{S}$ is $\lambda$-distributive, we may assume that $s$ decides the value of $\dom(\dot{p})$. Suppose first that $\zeta$ is a successor ordinal, and let $\zeta = \xi + 1$. If $\xi \not\in \dom(\dot{p})$, then we can find $u \leq (s, \dot{p}, \dot{t})$ in $\bb{U}_\xi$, and we are done. Thus, suppose $\xi \in \dom(\dot{p})$. Find $(s^*, \dot{p}', \dot{t}^*) \leq (s, \dot{p} \restriction \xi, \dot{t})$ such that:
	\begin{enumerate}
		\item{$(s^*, \dot{p}', \dot{t}^*) \in \bb{U}_\xi$.}
		\item{There is $c \in V$ such that $(s^*, \dot{p}') \Vdash ``\dot{p}(\xi) = c."$}
		\item{There is $\gamma \geq \max(c)$ such that $(s^*, \dot{p}', \dot{t}^*) \Vdash ``\gamma \in \dot{C}_\xi"$.}
	\end{enumerate}
	This is possible, because $\bb{U}_\xi$ is a dense, $\lambda$-closed subset of $\bb{S} * \dot{\bb{P}}_\xi * \dot{\bb{T}}$. Now, define $\dot{p}^* \in \dot{\bb{P}}_\zeta$ by letting $\dot{p}^* \restriction \xi = \dot{p}'$ and letting $\dot{p}^*(\xi)$ be a name forced by $(s^*, \dot{p}')$ to be equal to $\dot{p}(\xi) \cup \{\gamma \}$. It is easy to see that $(s^*, \dot{p}^*, \dot{t}^*) \leq (s, \dot{p}, \dot{t})$ and $(s^*, \dot{p}^*, \dot{t}^*) \in \bb{U}_\zeta$.

	Finally, suppose $\zeta$ is a limit ordinal. If $\cf(\zeta) = \lambda$, then $\dom(\dot{p})$ is bounded below $\zeta$ and we are done by the inductive hypothesis. Thus, assume that $\kappa := \cf(\zeta) < \lambda$. Let $\langle \zeta_i \mid i < \kappa \rangle$ be an increasing, continuous sequence of ordinals cofinal in $\zeta$. We will construct a sequence $\langle (s_i, \dot{p}_i, \dot{t}_i) \mid i < \kappa \rangle$ such that:
	\begin{enumerate}
		\item{For every $i < \kappa$, $(s_i, \dot{p}_i, \dot{t}_i) \in \bb{U}_{\zeta_i}$.}
		\item{For every $i < j < \kappa$, $(s_j, \dot{p}_j, \dot{t}_j) \leq (s_i, \dot{p}_i, \dot{t}_i)$.}
		\item{For every $i < \kappa$, $(s_i, \dot{p}_i {^\frown} \dot{p} \restriction [\zeta_i, \zeta), \dot{t}_i) \leq (s, \dot{p}, \dot{t})$.}
	\end{enumerate}
	The construction is straightforward, by recursion on $i$. Let $(s_0, \dot{p}_0, \dot{t}_0) \leq (s, \dot{p} \restriction \zeta_0, \dot{t})$ be in $\bb{U}_{\zeta_0}$. If $i = k + 1$, let $(s_i, \dot{p}_i, \dot{t}_i) \leq (s_k, \dot{p}_k {^\frown} \dot{p} \restriction [\zeta_k, \zeta_i), \dot{t}_k)$ be in $\bb{U}_{\zeta_i}$. If $i$ is a limit ordinal, note that $\langle (s_k, \dot{p}_k, \dot{t}_k) \mid k < i \rangle$ is a decreasing sequence of conditions in $\bb{U}_{\zeta_i}$ and thus has a lower bound in $\bb{U}_{\zeta_i}$. Let $(s_i, \dot{p}_i, \dot{t}_i)$ be such a lower bound. Finally, at the end of the construction, $\langle (s_i, \dot{p}_i, \dot{t}_i) \mid i < \kappa \rangle$ is a decreasing sequence of conditions in $\bb{U}_\zeta$, so, by the $\lambda$-closure of $\bb{U}_\zeta$, it has a lower bound in $\bb{U}_\zeta$. Let $(s^*, \dot{p}^*, \dot{t}^*)$ be such a lower bound. It is easily verified that $(s^*, \dot{p}^*, \dot{t}^*) \leq (s, \dot{p}, \dot{t})$, thus completing the proof.
\end{proof}

\section{Dense bounded stationary reflection} \label{densesect}

In this section, we construct a model in which there are singular cardinals $\delta < \mu$ such that, letting $\lambda = \mu^+$, $\mathrm{Refl}(\lambda)$ holds, and, for every stationary $S \subset S^\lambda_{<\delta}$, there is a stationary $T \subseteq S$ such that $T$ does not reflect at any ordinal in $S^\lambda_{>\delta}$. In our model, we arrange so that $\mu = \aleph_{\omega \cdot 2}$ and $\delta = \aleph_\omega$, though the technique is quite flexible.

\begin{theorem}
	Suppose there is a sequence of supercompact cardinals of order type $\omega \cdot 2$. Then there is a forcing extension in which $\mathrm{Refl}(\aleph_{\omega \cdot 2 + 1})$ holds and, for every stationary $S \subseteq S^{\aleph_{\omega \cdot 2 + 1}}_{<\aleph_\omega}$, there is a stationary $T\subseteq S$ such that $T$ does not reflect at any ordinals in $S^{\aleph_{\omega \cdot 2 + 1}}_{>\aleph_\omega}$.
\end{theorem}

\begin{proof}
	Assume GCH. Let $\langle \kappa_i \mid i \leq \omega \cdot 2 + 1 \rangle$ be an increasing, continuous sequence of cardinals such that:
	\begin{itemize}
		\item{$\kappa_0 = \omega$.}
		\item{If $i$ is $0$ or a successor ordinal, then $\kappa_{i+1}$ is supercompact.}
		\item{If $i$ is a limit ordinal, then $\kappa_{i+1} = \kappa_i^+$.}
	\end{itemize}
	For ease of notation, let $\lambda$ denote $\kappa_{\omega \cdot 2 + 1}$, $\mu = \kappa_{\omega \cdot 2}$, and $\delta = \kappa_\omega$. The reason for our numbering is that, in the final extension, we will have $\kappa_i = \aleph_i$ for all $i \leq \omega \cdot 2 + 1$.

	Let $\langle \bb{P}_i, \dot{\bb{Q}}_j \mid i \leq \omega \cdot 2, j < \omega \cdot 2 \rangle$ be a forcing iteration, taken with full supports, in which, if $i < \omega \cdot 2$ and $i \not= \omega$, then $\Vdash_{\bb{P}_i} ``\dot{\bb{Q}}_i = \dot{\mathrm{Coll}}(\kappa_i, < \kappa_{i + 1})"$ and, if $i = \omega$, $\dot{Q}_i$ is a $\bb{P}_i$ name for trivial forcing. Let $\bb{P} = \bb{P}_{\omega \cdot 2}$. Standard arguments (see, e.g. \cite{reflection}) show that, if $G$ is $\bb{P}$-generic over $V$, then, for all $i \leq \omega \cdot 2 + 1$, $\kappa_i = (\aleph_i)^{V[G]}$. Let $\dot{\vec{a}}$ be a $\bb{P}$-name for an enumeration, of order type $\lambda$, of all bounded subsets of $\lambda$, and let $\dot{\bb{A}}$ be a $\bb{P}$-name for $\bb{A}_{\dot{\vec{a}}}$, the forcing poset to shoot a club through the set of ordinals below $\lambda$ that are approachable with respect to $\dot{\vec{a}}$.

	Denote $V^{\bb{P}*\dot{\bb{A}}}$ by $V_1$. In $V_1$, we will define a forcing iteration $\langle \bb{S}_\xi, \dot{\bb{T}}_\zeta \mid \xi \leq \lambda^+, \zeta < \lambda^+ \rangle$. The iteration will use supports of size $\leq \mu$. For $\zeta \leq \lambda^+$, let $A_\zeta = \{\eta < \zeta \mid \eta$ is even\}. The definition of $\dot{\bb{T}}_\zeta$ will depend on whether $\zeta$ is even or odd. If $\zeta < \lambda^+$ is even (including limit ordinals), then choose an $\bb{S}_\zeta$-name $\dot{T}_\zeta$ for a stationary subset of $S^\lambda_{<\delta}$, and let $\dot{\bb{T}}_\zeta$ be an $\bb{S}_\zeta$-name for $\bb{S}_{\dot{T}_\zeta, \delta^+}$, i.e. the forcing to add a subset of $\dot{T}_\zeta$ that does not reflect at any points in $S^\lambda_{> \delta}$. Let $\dot{S}_\zeta$ be an $\bb{S}_{\zeta + 1}$-name for this subset of $\dot{T}_\zeta$, and let $S_\zeta$ denote its realization in $V_1^{\bb{S}_{\zeta + 1}}$. 

	If $\zeta \leq \lambda^+$, then, in $V_1^{\bb{S}_\zeta}$, let $\bb{U}_\zeta$ be the product $\prod_{\eta \in A_\zeta} CU(S_\eta)$, where the product is taken with supports of size $\leq \mu$. If $\zeta < \lambda^+$ is odd, we will choose an $\bb{S}_\zeta$-name $\dot{T}_\zeta$ for a subset of $S^\lambda_{>\delta}$ such that $\Vdash_{\bb{S}_\zeta * \dot{\bb{U}}_\zeta}``\dot{T}_\zeta$ is non-stationary," and let $\dot{\bb{T}}_\zeta$ be an $\bb{S}_\zeta$-name for $CU(\dot{T}_\zeta).$ Also, fix an $\bb{S}_\zeta * \dot{\bb{U}}_\zeta$-name $\dot{C}_\zeta$ for a club in $\lambda$ disjoint from $\dot{T}_\zeta$.

	Before we discuss the choice of the name $\dot{T}_\zeta$, we describe some of the properties of $\bb{S} := \bb{S}_{\lambda^+}.$ First note that, by a standard $\Delta$-system argument, $\bb{S}$ has the $\lambda^+$-c.c. Also, $\bb{S}$ is easily seen to be $\delta^+$-directed closed. We also claim that it is $\lambda$-distributive. To show this, we define another poset. In $V_1$, for all $\xi \leq \lambda^+$, let $\bb{V}_\xi$ be the set of $(s, u)$ such that:
	\begin{itemize}
		\item{$s \in \bb{S}_\xi$.}
		\item{$u$ is a function and $\dom(u) = \dom(s) \cap A_\xi$.}
		\item{For all $\zeta \in \dom(u)$, $u(\zeta)$ is a closed, bounded subset of $\lambda$ such that $s \restriction (\zeta + 1) \Vdash ``u(\zeta) \cap \dot{S}_\zeta = \emptyset"$ and, if $\alpha \in \mathrm{nacc}(u(\zeta))$, then $\cf(\alpha) = \omega$.}
	\end{itemize} 
	If $(s_0, u_0), (s_1, u_1) \in \bb{V}_\xi$, we let $(s_1, u_1) \leq (s_0, u_0)$ iff $s_1 \leq s_0$ in $\bb{S}_\xi$ and, for every $\zeta \in \dom(u_0)$, $u_1(\zeta)$ end-extends $u_0(\zeta)$. If $\zeta < \xi \leq \lambda^+$, the map $(s, u) \mapsto (s \restriction \zeta, u \restriction \zeta)$ defines a projection from $\bb{V}_\xi$ to $\bb{V}_\zeta$.

	\begin{lemma}
		For all $\xi \leq \lambda^+$,
		\begin{enumerate}
			\item{For every $(s,u) \in \bb{V}_\xi$, there is $(s^*, u^*) \leq (s, u)$ such that:
			\begin{enumerate}
				\item{For every $\zeta \in \dom(s^*)$, there is $s_\zeta \in V_1$ such that $s^* \restriction \zeta \Vdash ``s^*(\zeta) = s_\zeta."$}
				\item{For every $\zeta \in \dom(s^*) \cap A_\zeta$, $\gamma_{s^*(\zeta)} = \max(u(\zeta))$}
				\item{For every $\zeta \in \dom(s^*) \setminus A_\zeta$, $(s^* \restriction \zeta, u^* \restriction \zeta) \Vdash_{\bb{V}_\zeta}``\max(s^*(\zeta)) \in \dot{C}_\zeta."$ (Note that this will make sense if clause (3) below holds at $\zeta$.)}
			\end{enumerate}}
			\item{$\bb{S}_\xi$ is $\lambda$-distributive.}
			\item{$\bb{V}_\xi$ is isomorphic to a dense subset of $\bb{S}_\xi * \dot{\bb{U}}_\xi$.}
		\end{enumerate}
	\end{lemma}

	\begin{proof}
		We prove all three statements simultaneously by induction on $\xi.$ First note that, to show (3), it suffices to show that, for every $(s, \dot{u}) \in \bb{S}_\xi * \dot{\bb{U}}_\xi$, there is $s^* \leq s$ and $u^* \in V_1$ such that $s^* \Vdash ``\dot{u} = u^*"$. This is easily implied by (2), as $\dot{u}$ can be thought of as a name for a set of pairs of ordinals of size $<\lambda$. Also note that the set of $(s^*, u^*) \in \bb{V}_\xi$ as given in the conclusion of (1) is easily seen to be $\lambda$-directed closed so, since $(s, u) \mapsto s$ is a projection from $\bb{V}_\xi$ to $\bb{S}_\xi$, (1) implies (2) for a fixed $\xi \leq \lambda^+$.

		Fix $\xi \leq \lambda^+$. Assume we have proven all three statements for all $\zeta < \xi$. We prove (1) for $\xi$. Assume first that $\xi$ is a successor ordinal, with $\xi = \zeta + 1$ and $\zeta$ odd. Let $(s, u) \in \bb{V}_\xi$. If $\zeta \not\in \dom(s)$, then we are done by (1) for $\zeta$. Thus, assume $\zeta \in \dom(s)$. Since $\bb{S}_\zeta$ is $\lambda$-distributive, we can find $s' \leq s \restriction \zeta$ and $s_\zeta \in V_1$ such that $s' \Vdash ``s(\zeta) = s_\zeta."$. Now find $(\bar{s}, \bar{u}) \leq (s', u)$ in $\bb{V}_\zeta$ and $\alpha > \max(s_\zeta)$ such that $\cf(\alpha) = \omega$ and $(\bar{s}, \bar{u}) \Vdash ``\alpha \in \dot{C}_\zeta."$ Form $(s^*, u^*) \leq (s, u)$ by letting $(s^* \restriction \zeta, u^*) \leq (\bar{s}, \bar{u})$ witness (1) for $\zeta$ and by letting $s^*(\zeta)$ be a name forced by $s^* \restriction \zeta$ to be equal to $s_\zeta \cup \{\alpha\}.$ It is easily verified that $(s^*, u^*)$ is as desired.

		Next, suppose that $\xi = \zeta + 1$ and $\zeta$ is even. Let $(s, u) \in \bb{V}_\xi$, and again assume that $\zeta \in \dom(s).$ Find $s' \leq s \restriction \zeta$ and $s_\zeta \in V_1$ such that $s' \Vdash ``s(\zeta) = s_\zeta."$ Find $\alpha$ with $\max(u(\zeta)), \gamma_{s_\zeta} < \alpha < \lambda$ and $\cf(\alpha) = \omega$. Form $(s^*, u^*) \leq (s, u)$ by letting $(s^* \restriction \zeta, u^* \restriction \zeta) \leq (s', u \restriction \zeta)$ witness (1) for $\zeta$, letting $s^*(\zeta)$ be a name forced by $s^* \restriction \zeta$ to be a function in ${^{\alpha + 1}}2$ such that $s^*(\zeta) \restriction (\gamma_{s_\zeta} + 1) = s_\zeta$ and $s^*(\zeta)(\beta) = 0$ for all $\beta \in (\gamma_{s_\zeta}, \alpha]$, and letting $u^*(\zeta) = u \cup \{\alpha\}$.

		Finally, suppose $\xi$ is a limit ordinal, and let $(s, u) \in \bb{V}_\xi$. If $\dom(s)$ is bounded below $\xi$ (in particular, if $\cf(\xi) \geq \lambda$), then we are done by the induction hypothesis. Thus, assume $\cf(\xi) < \lambda$ and $\dom(s)$ is unbounded in $\xi$. Let $\langle \xi_i \mid i < \cf(\xi) \rangle$ be an increasing, continuous sequence of ordinals cofinal in $\xi$. Form a sequence of conditions $\langle (s_i, u_i) \mid i < \cf(\xi) \rangle$ such that:
		\begin{itemize}
			\item{For all $i < \cf(\xi)$, $(s_i, u_i) \in \bb{V}_{\xi_i}$ and $(s_i, u_i) \leq (s \restriction \xi_i, u \restriction \xi_i)$.}
			\item{For all $i < k < \cf(\xi)$, $(s_k \restriction \xi_i, u_k \restriction \xi_i) \leq (s_i, u_i)$.}
			\item{For all $i < \cf(\xi)$, $(s_i, u_i)$ witnesses (1) for $\xi_i$.}
		\end{itemize}
		The construction is a straightforward recursion; we omit the details. At the end of the construction, let $X = \bigcup_{i < \cf(\xi)} \dom(s_i)$. For $\zeta \in X$, let $i_\zeta$ be the least $i$ such that $\zeta \in \dom(s_i)$. For $i_\zeta \leq i < \cf(\xi)$, let $\alpha_{\zeta, i}$ be such that, if $\zeta$ is even, then $s_i \restriction \zeta \Vdash ``\gamma_{s_i(\zeta)} = \alpha_{\zeta, i}"$ and, if $\zeta$ is odd, then $s_i \restriction \zeta \Vdash ``\max(s_i(\zeta)) = \alpha_{\zeta, i}."$ Let $\alpha_\zeta = \sup(\{\alpha_{\zeta, i} \mid i_\zeta \leq i < \cf(\xi)\})$. Form $(s^*, u^*)$ as follows. $\dom(s^*) = X$. For $\zeta \in X$, let $s^*(\zeta)$ be a name forced by $s^* \restriction \zeta$ to be equal to $\{(\alpha_\zeta, 0)\} \cup \bigcup_{i_\zeta \leq i < \cf(\xi)} s_i(\zeta)$ if $\zeta$ is even and $\{\alpha_\zeta\} \cup \bigcup_{i_\zeta \leq i < \cf(\xi)} s_i(\zeta)$ if $\zeta$ is odd. For $\zeta \in A_\xi \cap X$, let $u^*(\zeta) = \{\alpha_\zeta\} \cup \bigcup_{i_\zeta \leq i < \cf(\xi)} u_i(\zeta).$ $(s^*, u^*)$ is easily seen to be as required by (1). 
	\end{proof}

	Note that the above proof also shows that, in $V^{\bb{P}*\dot{\bb{A}}}$, $\bb{S} * \dot{\bb{U}}$ has a dense, $\lambda$-directed-closed subset.

	Since, in $V_1$, $\bb{S}$ has the $\lambda^+$-c.c., we can, by standard bookkeeping arguments, assume that we chose our names $\dot{T}_\zeta$ so that, in $V_1^{\bb{S}}$, every stationary subset of $S^\lambda_{<\delta}$ was considered as $\dot{T}_\zeta$ for some even $\zeta < \lambda^+$ and every subset of $S^\lambda_{>\delta}$ was considered as $\dot{T}_\zeta$ for cofinally many odd $\zeta < \lambda^+$. In $V_1^{\bb{S}}$, let $A = A_{\lambda^+}$ and $\bb{U} = \bb{U}_{\lambda^+}$. By the distributivity of $\bb{S}$, all conditions of $\bb{U}$ are in $V_1$. 

	Let $G$ be $\bb{P}$-generic over $V$, let $H$ be $\bb{A}$-generic over $V[G]$, and let $I$ be $\bb{S}$-generic over $V[G*H]$. $V[G*H*I]$ will be our final model. For $i < \omega \cdot 2$, let $G_i$ be the $\bb{P}_i$-generic filter induced by $G$, and, for $\zeta < \lambda^+$, let $I_\zeta$ be the $\bb{S}_\zeta$-generic filter induced by $I$. Note that, in $V[G*H]$ (and hence in all further extensions preserving $\lambda$), $AP_\mu$ holds. Thus, by Lemma \ref{genstat}, for all even $\zeta < \lambda^+$, $S_\zeta$ (the stationary set added by $\bb{T}_\zeta$) is stationary in $V[G*H*I_{\zeta + 1}]$. Because the remainder of the iteration is $\delta^+$-closed and $S \subseteq S^\lambda_{<\delta}$, $S_\zeta$ is stationary in $V[G*H*I]$ by Fact \ref{apstat}. We have therefore arranged so that, in $V[G*H*I]$, for every stationary $T\subseteq S^\lambda_{<\delta}$, there is a stationary $S\subseteq T$ that does not reflect at any ordinals in $S^\lambda_{>\delta}$. Also, suppose that, in $V[G*H*I]$, $T \subseteq S^\lambda_{>\delta}$ and $\Vdash_{\bb{U}}``T$ is non-stationary.$"$ By standard chain condition arguments, there is $\xi < \lambda^+$ such that $T \in V[G*H*I_\xi]$ and, in $V[G*H*I_\xi]$, $\Vdash_{\bb{U}_\xi}``T$ is non-stationary.$"$ Thus, by our bookkeeping for the choice of $\dot{T}_\zeta$ at odd $\zeta$, we have that $T$ is already non-stationary in $V[G*H*I]$.

	We now argue that $\mathrm{Refl}(\lambda)$ holds in $V[G*H*I]$. So, let $T\subseteq \lambda$ be stationary in $V[G*H*I]$. We can assume, by shrinking $T$ if necessary, that there is $i_0 < \omega \cdot 2$ such that $T \subseteq S^\lambda_{\kappa_{i_0}}$. We first consider the case in which $i_0 < \omega$. Let $i^* = i_0 + 2$. In $V[G_{i_0+1}]$, $\kappa_{i^*}$ remains supercompact. Fix an elementary embedding $j:V[G_{i_0+1}] \rightarrow M[G_{i_0+1}]$ witnessing that $\kappa_{i^*}$ is $\lambda^+$-supercompact. In $V[G_{i_0+1}]$, $\bb{Q}_{i_0+1}$ = $\mathrm{Coll}(\kappa_{i_0 + 1}, < \kappa_{i^*})$ and $j(\bb{Q}_{i_0+1}) = \mathrm{Coll}(\kappa_{i_0+1}, < j(\kappa_{i^*}))$. Since $\bb{P}_{i^*, \omega \cdot 2} * \dot{\bb{A}} * \dot{\bb{S}}$ is strongly $\kappa_{i_0 + 1}$-strategically closed, we can apply Fact \ref{lift} to observe that $j(\bb{Q}_{i_0+1}) \cong \bb{P}_{i_0+1, \omega \cdot 2} * \dot{\bb{A}} * \dot{\bb{S}} * \dot{\bb{R}}$, where $\dot{\bb{R}}$ is forced to be $\kappa_{i_0+1}$-closed. Thus, letting $J$ be $\bb{R}$-generic over $V[G*H*I]$, we can extend $j$ to $j:V[G_{i^*}] \rightarrow M[G*H*I*J]$.

	We would like to extend $j$ further to have domain $V[G*H*I]$. To do this, we define a master condition $(p^*, \dot{a}^*, \dot{s}^*) \in j(\bb{P}_{i^*, \omega \cdot 2} * \dot{\bb{A}} * \dot{\bb{S}})$ in $M[G*H*I*J]$, i.e. a condition $(p^*, \dot{a}^*, \dot{s}^*)$ such that, for all $(p, \dot{a}, \dot{s}) \in G_{i^*, \omega \cdot 2} * H * I$, $(p^*, \dot{a}^*, \dot{s}^*) \leq j((p, \dot{a}, \dot{s}))$. The definition is straightforward. Let $\eta = \sup(j``\lambda)$. For $i^* \leq i < \omega \cdot 2$, we let $p^*(i)$ be a name for $\bigcup_{p \in G} j(p(i))$. $p^* \in j(\bb{P}_{i^*, \omega \cdot 2})$ by the fact that $j(\bb{P}_{i^*, \omega \cdot 2})$ is $j(\kappa_{i^*})$-directed closed. We then define $\dot{a}^*$ to be a name forced by $p^*$ to be equal to $\{\eta \} \cup \bigcup_{a \in H} j(a)$. The only thing to check here is that $\eta$ is forced to be approachable with respect to $j(\vec{a})$. The proof of this can be found, among other places, in \cite{reflection}. Finally, since $j(\dot{\bb{S}})$ is forced to be $j(\delta^+)$-directed-closed, it is straightforward to find a name $\dot{s}^*$ forced by $(p^*, \dot{a}^*)$ to be a lower bound for $\{ j(\dot{s}) \mid \dot{s} \in I \}$. $(p^*, \dot{a}^*, \dot{s}^*)$ is then as desired, and, letting $G^+ * H^+ * I^+$ be $j(\bb{P}_{i^*, \omega \cdot 2} * \dot{\bb{A}} * \dot{\bb{S}})$-generic over $V[G*H*I*J]$ with $(p^*, \dot{a}^*, \dot{s}^*) \in G^+ * H^+ * I^+$, we can extend $j$ to $j:V[G*H*I] \rightarrow M[G*H*I*J*G^+*H^+*I^+]$. Now, by standard arguments (see e.g. Proposition 1.1 in \cite{reflection}), if $T$ does not reflect in $V[G*H*I]$, then $j``T$ is non-stationary in $\eta$ in $M[G*H*I*J*G^+*H^+*I^+]$, which further implies that $T$ is non-stationary $V[G*H*I*J*G^+*H^+*I^+]$. Since $G^+*H^+*I^+$ is generic for $<j(\kappa_{i^*})$-strategically-closed forcing, it could not have added any new subsets of $\lambda$, thus $T$ is already non-stationary in $V[G*H*I*J]$. However, since $J$ is generic for $\kappa_{i^*}$-closed forcing, $T\subseteq S^\lambda_{<\kappa_{i^*}}$, and $AP_\mu$ holds in $V[G*H*I]$, Fact \ref{apstat} implies that $T$ is non-stationary in $V[G*H*I]$, which is a contradiction. Thus, $T$ reflects in $V[G*H*I]$.

	Next, suppose $i_0 > \omega$. Again, let $i^* = i_0 + 2$ and let $j:V[G_{i_0 + 1}] \rightarrow M[G_{i_0 + 1}]$ witness that $i^*$ is $\lambda^+$-supercompact. $\bb{P}_{i^*, \omega \cdot 2} * \dot{\bb{A}} * \dot{\bb{S}} * \dot{\bb{U}}$ has a dense strongly $\kappa_{i_0 + 1}$-strategically closed subset, so, again applying Fact \ref{lift}, $j(\bb{Q}_{i_0+1}) \cong \bb{P}_{i^*, \omega \cdot 2} * \dot{\bb{A}} * \dot{\bb{S}} * \dot{\bb{U}} * \dot{\bb{R}}$, where $\dot{\bb{R}}$ is fored to be $\kappa_{i_0 + 1}$-closed. Note that, by previous arguments, it is not the case that, in $V[G*H*I]$, $\Vdash_{\bb{U}}``T$ is non-stationary$"$. Thus, letting $J$ be $\bb{U}$-generic over $V[G*H*I]$ such that $T$ remains stationary in $V[G*H*I*J]$, and letting $K$ be $\bb{R}$-generic over $V[G*H*I*J]$, we can lift $j$ to $j:V[G_{i^*}] \rightarrow M[G*H*I*J*K]$. We can extend $j$ further to $j:V[G*H*I] \rightarrow M[G*H*I*J*K*G^+*H^+*I^+]$ using a master condition argument as in the previous case, exploiting the fact that $\bb{S}*\dot{\bb{U}}$ has a dense $\lambda$-closed subset in $V^{\bb{P}*\dot{\bb{A}}}$. And again, exactly as in the previous case, we can argue that $T$ must reflect in $V[G*H*I]$, for otherwise it would be non-stationary in $V[G*H*I*J]$.
\end{proof}

\section{Global bounded stationary reflection} \label{globalsect}

In this section, we improve upon results from \cite{reflection}, producing, from large cardinal assumptions, a model in which bounded stationary reflection holds at every possible successor of a singular cardinal.

\begin{theorem}
	Suppose there is a proper class of supercompact cardinals. Then there is a class forcing extension in which, for every singular cardinal $\mu > \aleph_\omega$, $\mathrm{Refl}(\mu^+)$ holds and there is a stationary subset of $S^{\mu^+}_\omega$ that does not reflect in $S^{\mu^+}_{>\aleph_\omega}$.
\end{theorem}

\begin{proof}
	Assume GCH. Let $\langle \kappa_i \mid i \in \mathrm{On} \rangle$ be an increasing, continuous sequence of cardinals such that:
	\begin{itemize}
		\item{$\kappa_0 = \omega$.}
		\item{If $i$ is a limit ordinal or a successor of a limit ordinal, then $\kappa_{i+1} = \kappa_i^+$.}
		\item{If $i$ is not a limit ordinal or a successor of a limit ordinal, then $\kappa_{i+1}$ is supercompact.}
	\end{itemize}
	We may assume that, if $i$ is a limit ordinal, then $\kappa_i$ is singular by cutting the universe off at the least regular $\kappa_i$ with $i$ limit.

	We define a class forcing iteration $\langle \bb{P}_i, \dot{\bb{Q}}_i \mid i \in \mathrm{On} \rangle$, taken with full supports. If $i = 0$, $i = 1$, or $i$ is a successor of a successor ordinal, then let $\dot{\bb{Q}}_i$ be such that $\Vdash_i ``\dot{\bb{Q}}_i = \mathrm{Coll}(\kappa_i, < \kappa_{i+1})."$ If $i = \omega$ or $i$ is a successor of a limit ordinal, let $\dot{\bb{Q}}_i$ be a $\bb{P}_i$-name for trivial forcing.

	It remains to define $\dot{\bb{Q}}_i$ when $i > \omega$ is a limit ordinal. Fix such an $i$, and move temporarily to $V^{\bb{P}_i}$. Let $\vec{a}_i$ be an enumeration of the bounded subsets of $\kappa_{i+1}$ in order type $\kappa_{i+1}$, and let $\bb{A}_i$ be the poset to shoot a club through the set of ordinals below $\kappa_{i+1}$ that are approachable with respect to $\vec{a}_i$. In $V^{\bb{P}_i * \dot{\bb{A}}_i}$, let $\bb{S}_i$ be $\bb{S}^{\kappa_{i+1}}_{\omega, \kappa_{\omega + 1}}$. In $V^{\bb{P}_i * \dot{\bb{A}}_i * \dot{\bb{S}}_i}$, we will define an iteration, $\langle \bb{T}^i_\xi, \dot{\bb{U}}^i_\zeta \mid \xi \leq \kappa^+_{i+1}, \zeta < \kappa^+_{i+1} \rangle$, taken with supports of size $\kappa_i$ and, letting $\bb{T}_i = \bb{T}^i_{\kappa^+_{i+1}},$ we will let $\dot{\bb{Q}}_i$ be a $\bb{P}_i$-name for $\bb{A}_i * \dot{\bb{S}}_i * \dot{\bb{T}}_i$. If $p \in \bb{P}_{i+1}$, we will let $p(i)_0$, $p(i)_1$, and $p(i)_2$ denote the $\bb{A}_i$, $\dot{\bb{S}}_i$, and $\dot{\bb{T}}_i$ parts of $p(i)$, respectively. If $\zeta < \kappa^+_{i+1}$, we will let $p(i) \restriction \zeta$ denote $(p(i)_0, p(i)_1, p(i)_2 \restriction \zeta)$. Moreover, for $\zeta < \kappa^+_{i+1},$ and $k < i$ we will let $\bb{P}_{k, i+1} \restriction \zeta$ denote $\bb{P}_{k,i} * \dot{\bb{A}}_i * \dot{\bb{S}}_i * \dot{\bb{T}}^i_\zeta$.

	Suppose $\zeta < \kappa^+_{i+1}$ and we have defined $\bb{T}^i_\zeta$. We describe how to define $\dot{\bb{U}}^i_\zeta$. For all limit ordinals $\omega < i' \leq i$, let $S_{i'}$ be the stationary subset of $S^{\kappa_{i'+1}}_\omega$ added by $\bb{S}_{i'}$. For all $\omega < k < i$, with $k$ a successor ordinal, let $X^i_k$ be the set of limit ordinals in $(k,i]$, and let $\bb{C}^\zeta_{k,i}$ be the poset defined in $V^{\bb{P}_k}$ as follows. Conditions are pairs $(p, c)$ such that:
	\begin{itemize}
		\item{$p \in \bb{P}_{k,i+1} \restriction \zeta$.}
		\item{$c$ is a function, and $\dom(c) = X^i_k$.}
		\item{For all $i' \in \dom(c)$, $c(i')$ is a $\bb{P}_{k, i'}$-name for a closed, bounded subset of $\kappa_{i'+1}$ and $p \restriction [k, i')^\frown p(i')_0 {^\frown} p(i')_1 \Vdash ``c(i') \cap \dot{S}_{i'} = \emptyset."$}
	\end{itemize}
	$(p',c') \leq (p, c)$ if $p' \leq p$ and, for all $i' \in \dom(c)$, $p' \restriction i' \Vdash ``c'(i')$ end-extends $c(i')$." 

	The map $(p,c) \mapsto p$ is clearly a projection from $\bb{C}^\zeta_{k,i}$ to $\bb{P}_{k, i+1} \restriction \zeta$. Let $\bb{V}^\zeta_{k,i} \in V^{\bb{P}_{i+1} \restriction \zeta}$ be the quotient poset, so $\bb{C}^\zeta_{k,i} \cong \bb{P}_{k,i+1} \restriction \zeta * \dot{\bb{V}}^\zeta_{k,i}$. Let $\dot{T}^i_\zeta$ be a $\bb{T}^i_\zeta$-name for a subset of $S^{\kappa_{i+1}}_{>\kappa_\omega}$ such that, for every successor ordinal $k$ with $\omega < k < i$, $\Vdash_{\bb{V}^\zeta_{k,i}}``\dot{T}^i_\zeta$ is non-stationary," and let $\dot{\bb{U}}^i_\zeta$ be a $\bb{T}^i_\zeta$-name for $CU(\dot{T}^i_\zeta)$. If $k$ is a successor ordinal with $\omega < k < i$, notice that, if $\zeta < \zeta' \leq \kappa^+_{i+1}$, there is a natural projection from $\bb{C}^{\zeta'}_{k,i}$ to $\bb{C}^{\zeta}_{k,i}$. Let $\bb{C}_{k,i} = \bb{C}^{\kappa^+_{i+1}}_{k,i}$, and let $\bb{V}_{k,i}$ be the quotient forcing over $\bb{P}_{k,i+1}$ in $V^{\bb{P}_{i+1}}$. Notice also that, if $(p,c) \in \bb{C}_{k,i}$, then there is $\zeta < \kappa^+_{i+1}$ such that $(p,c) \in \bb{C}^\zeta_{k,i}$ and that, if $i' \in (k,i)$ is a limit ordinal, then $\{p \restriction (i'+1), c \restriction (i'+1)| \mid (p,c) \in \bb{C}^\zeta_{k,i} \} = \bb{C}_{k, i'}$ and that $\bb{C}^\zeta_{k,i} \cong \bb{C}_{k, i'} * \dot{\bb{C}}^\zeta_{i'+1, i}.$

	Note that, in $V^{\bb{P}_i * \dot{\bb{A}}_i * \dot{\bb{S}}_i}$, $\bb{T}_i$ has the $\kappa^+_{i+1}$-c.c., so, by standard bookkeeping arguments, we can arrange so that every canonical $\bb{T}_i$-name for a subset of $S^{\kappa_{i+1}}_{>\kappa_\omega}$ was considered as $\dot{T}^i_\zeta$ for cofinally many $\zeta < \kappa^+_{i+1}$. Thus, we can arrange that, in $V^{\bb{P}_{i+1}}$, if $T \subseteq S^{\kappa_{i+1}}_{>\kappa_\omega}$ is such that, for every successor ordinal $k < i$, $\Vdash_{\bb{V}_{k,i}}``T$ is non-stationary," then $T$ is already non-stationary in $V^{\bb{P}_{i+1}}$.

	\begin{lemma} \label{clemma}
		Let $\omega < k < i$, with $k$ a successor ordinal and $i$ a limit ordinal, and let $\zeta \leq \kappa^+_{i+1}$. In $V^{\bb{P}_k}$, for every $\ell \in X^i_k$ and every $\xi < \kappa^+_{\ell+1}$, let $\dot{C}^{\xi}_\ell$ be a $\bb{C}^\xi_{k, \ell}$-name for a club in $\kappa_{\ell + 1}$ disjoint from $\dot{T}^\ell_\xi$.
		\begin{enumerate}
			\item{For every $(p, c) \in \bb{C}^\zeta_{k,i}$, there is $(p^*, c^*) \leq (p,c)$ such that:
			\begin{enumerate}
				\item{There is $h \in \prod_{\ell \in X^i_k} \kappa_{\ell + 1}$ such that, for all $\ell \in X^i_k$, $p^* \restriction \ell {^\frown} p^*(\ell)_0 \Vdash ``\gamma_{p^*(\ell)_1} = h(\ell) = \max(c^*(\ell))."$}
				\item{For every $\ell \in X^i_k$, for every $\xi < \kappa^+_{\ell+1}$, (or $\xi < \zeta$ if $\ell = i$) $(p^* \restriction \ell {^\frown} p^*(\ell) \restriction \xi, c^* \restriction (\ell + 1)) \Vdash ``\xi \not\in \dom(p^*(\ell)_2)$ or $\max(p^*(\ell)_2(\xi)) \in \dot{C}^\xi_\ell.")$}
			\end{enumerate}}
			\item{$\bb{P}_{k, i+1} \restriction \zeta$ is $\kappa_k$-distributive.}
		\end{enumerate}
	\end{lemma}

	\begin{proof}
		First note that, by now-familiar arguments, the set of conditions $(p^*, c^*)$ that satisfy (1)(a) and (1)(b), which we will denote by $\bb{C}^{\zeta, *}_{k,i}$ (or just $\bb{C}^*_{k,i}$ if $\zeta = \kappa^+_{i+1}$), is easily seen to be strongly $\kappa_k$-strategically closed. This will be useful in the inductive proof of (1) and also immediately yields (2) from the corresponding instance of (1).

		We proceed by induction on $i$ and, for fixed $i$, by induction on $\zeta \leq \kappa^+_{i+1}$. Thus, let $k < i$ be given, let $\zeta = 0$, and let $(p,c) \in \bb{C}^\zeta_{k,i}$. First, suppose that $i = k + \omega$. In this case, find $p' \leq p \restriction [k,i) {^\frown} p(i)_0$ and $\alpha_i < \kappa_{i+1}$ such that $\cf(\alpha_i) = \omega$ and $p' \Vdash ``\gamma_{p(i)_1}, \max(c(i)) < \alpha_i."$. Form $(p^*, c^*)$ by letting $p^* \restriction i {^\frown} p^*(i)_0 = p'$, letting $p^*(i)_1$ be a name forced by $p'$ to be equal to the function in ${^{\alpha_i+1}2}$ extending $p(i)_1$ that is constantly zero on $(\gamma_{p(i)_1}, \alpha_i + 1)$, and letting $c^*(i)$ be a name forced by $p'$ to be equal to $c(i) \cup \{\alpha\}$.

		Next, suppose $i = i' + \omega$ for some limit ordinal $i' \in (k, i)$. Find $p' \leq p \restriction i {^\frown} p(i)_0$ and $\alpha_i < \kappa_{i+1}$ such that $\cf(\alpha_i) = \omega$ and $p' \Vdash ``\gamma_{p(i)_1}, \max(c(i)) < \alpha_i."$. Define $(p^*, c^*)$ by letting $(p^* \restriction [k,i'+1), c^* \restriction [k,i'+1)) \leq (p' \restriction [k, i'+1), c \restriction [k,i'+1))$ witness (1) for $\bb{C}_{k, i'}$, letting $p^* \restriction [i'+1, i) {^\frown} p^*(i)_0 = p' \restriction [i'+1, i) {^\frown} p'(i)_0$, letting $p^*(i)_1$ be a name forced by $p'$ to be equal to the function in ${^{\alpha_i+1}2}$ extending $p(i)_1$ that is constantly zero on $(\gamma_{p(i)_1}, \alpha_i + 1)$, and letting $c^*(i)$ be a name forced by $p'$ to be equal to $c(i) \cup \{\alpha\}$.

		Finally, suppose that $i$ is a limit of limit ordinals. We first suppose that $\cf(i) < \kappa_k$. Let $\langle i_\eta \mid \eta < \cf(i) \rangle$ be an increasing, continuous sequence of limit ordinals from $(k,i)$ that is cofinal in $i$. Find $p' \leq p \restriction i {^\frown} p(i)_0$ and $\alpha_i$ as in the previous cases. Recursively construct a sequence $\langle (p_\eta, c_\eta) \mid \eta < \cf(i) \rangle$ such that:
		\begin{itemize}
			\item{For all $\eta < \cf(i)$, $(p_\eta, c_\eta) \in \bb{C}_{k, i_\eta}$ and satisfies (1).}
			\item{For all $\eta < \cf(i)$, $(p_\eta, c_\eta) \leq (p' \restriction (i_\eta + 1), c \restriction (i_\eta + 1))$.}
			\item{For all $\eta < \eta' < \cf(i)$, $(p_{\eta'} \restriction (i_\eta + 1), c_{\eta'} \restriction (i_\eta + 1)) \leq (p_\eta, c_\eta)$.}
			\item{For all $i' \in X^i_k$, if $\eta^*$ is the least $\eta$ such that $i' \leq i_\eta$, then, for all $\eta^* \leq \eta < \cf(i)$, $p_\eta \restriction i'$ forces that $\langle p_{\delta}(i')_0 \mid \eta^* \leq \delta \leq \eta \rangle$ is a partial run of $G^*_{\kappa_k}(\bb{A}_{i'})$ with Player II playing according to her winning strategy.}
		\end{itemize} 
		The construction is straightforward by the induction hypothesis, and it is straightforward to use $\langle (p_\eta, c_\eta) \mid \eta < \cf(i) \rangle$ to, by taking unions (and, where appropriate, closures) along all coordinates and adding $\alpha_i$ to the end of the $i$th coordinates as in the previous cases, get a $(p^*, c^*)$ as desired.

		If $\kappa_k \leq \cf(i)$, then find $k < k' < i$ with $k'$ a limit ordinal and $\cf(i) < \kappa_{k'+1}$. Move temporarily to $V^{\bb{P}_{k'+1}}$, and interpret $(p \restriction [k'+1, i+1), c \restriction [k'+1, i+1))$ in $\bb{C}^\zeta_{k'+1, i}$ as $(p_0, c_0)$. For every $\ell \in X^i_{k'+1}$ and $\xi < \kappa^+_{\ell+1}$, the quotient forcing of $\bb{C}^\xi_{k, \ell}$ over $\bb{P}_{k, k'+1} * \bb{C}^\xi_{k'+1, \ell}$ has the $\kappa^+_{k'+1}$-c.c., so we can find a $\bb{C}^\xi_{k'+1, \ell}$-name $\dot{D}^\xi_\ell$ for a club in $\kappa_{\ell + 1}$ that is forced by the quotient forcing of $\bb{C}^\xi_{k, \ell}$ over $\bb{P}_{k, k'+1}$ to be a subset of $\dot{C}^\xi_\ell$. Use the argument from the previous paragraph to find $(p_1, c_1) \leq (p_0, c_0)$ satisfying (1) for $\bb{C}^\zeta_{k'+1, i}$ and the set of $\dot{D}^\xi_\ell$'s as witnessed by $h_1 \in \prod_{\ell \in X^i_{k'+1}} \kappa_{\ell + 1}$. Let $(\dot{p}_1, \dot{c}_1)$ and $\dot{h}_1$ be $\bb{P}_{k, k'+1}$-names for $(p_1, c_1)$ and $h_1$, respectively. Since $\bb{P}_{k, k'+1}$ satisfies the $\kappa_{k'+2}$-c.c., there is a function $h^*_1 \in \prod_{\ell \in X^i_{k'+1}} \kappa_{\ell + 1}$ in $V^{\bb{P}_k}$ such that $\Vdash ``\dot{h}_1 \leq h^*_1"$ and $\cf(h^*_1(\ell)) = \omega$ for all $\ell \in X^i_{k'+1}$. Find $p_2 \leq p \restriction [k, k'+1)$ forcing that $(\dot{p}_1, \dot{c}_1)$ satisfies (1) as witnessed by $h^*_1$. Now form $(p^*, c^*)$ by letting $(p^* \restriction [k, k'+1), c^* \restriction [k, k'+1)) \leq (p_2, c \restriction [k, k'+1))$ witness (1) for $\bb{C}_{k, k'}$ and letting $(p^* \restriction [k'+1, i+1), c^* \restriction [k'+1, i+1)) = (\dot{p}_1, \dot{c}_1).$

		We now deal with the case $\zeta > 0$. First, suppose $\zeta = \zeta_0 + 1$. We may assume that $p \restriction i {^\frown} p(i)_0 {^\frown} p(i)_1$ decides whether $\zeta_0 \in \dom(p(i)_2)$. If it decides $\zeta_0 \not\in \dom(p(i)_2)$, then we are done by the induction hypothesis applied to $\zeta_0$. Otherwise, find $(p', c') \leq (p \restriction i {^\frown} p(i) \restriction \zeta_0, c)$ and $\alpha < \kappa_{i+1}$ such that $\cf(\alpha) = \omega$ and $(p',c') \Vdash ``\max(p(i)_2(\zeta_0)) < \alpha$ and $\alpha \in \dot{C}^{\zeta_0}_i$." Form $(p^*, c^*)$ by letting $(p^* \restriction i {^\frown} p^*(i) \restriction \zeta_0, c^*) \leq (p',c')$ witness (1) for $\bb{C}^{\zeta_0}_{k,i}$ and letting $p^*(i)_2(\zeta_0)$ be a name forced by $p'$ to be equal to $p(i)_2(\zeta_0) \cup \{\alpha\}$.

		Suppose $\zeta$ is a limit ordinal. If $\cf(\zeta) \geq \kappa_{i+1}$, then, strenthening $p$ if necessary, we may assume $(p,c) \in \bb{C}^{\zeta_0}_{k,i}$ for some $\zeta_0 < \zeta$, and we are done by the induction hypothesis. Thus, suppose $\mu := \cf(\zeta) < \kappa_i$. Also assume that $\mu < \kappa_k$. If $\mu \geq \kappa_k$, the same trick we used in the $\zeta = 0$ case will work. Let $\langle \zeta_\eta \mid \eta < \mu \rangle$ be an increasing, continuous sequence of ordinals, cofinal in $\zeta$, and construct a sequence $\langle (p_\eta, c_\eta) \mid \eta < \mu \rangle$ such that:
		\begin{itemize}
			\item{For all $\eta < \mu$, $(p_\eta, c_\eta) \in \bb{C}^{\zeta_\eta}_{k,i}$ and satisfies (1).}
			\item{For all $\eta < \mu$, $(p_\eta, c_\eta) \leq (p \restriction i {^\frown} p(i) \restriction \zeta_\eta, c)$.}
			\item{For all $\eta < \eta' < \mu$, $(p_{\eta'} \restriction i {^\frown} p_{\eta'}(i) \restriction \zeta_\eta, c_{\eta'}) \leq (p_\eta, c_\eta)$.}
			\item{For all $\ell \in X^i_k$ and all $\eta < \mu$, $p_\eta \restriction \ell$ forces that the sequence $\langle p_\delta(\ell)_0 \mid \delta \leq \eta \rangle$ is a partial run of the game $G^*_{\kappa_k}(\bb{A}_{\ell})$ with Player II playing according to her winning strategy.}
		\end{itemize}
		The construction is a straightforward recursion, and, as in the $\zeta = 0$ case, it is easy to see that $\langle (p_\eta, c_\eta) \mid \eta < \mu \rangle$ gives us a $(p^*, c^*) \leq (p,c)$ witnessing (1).
	\end{proof}

	Note that $\bb{C}^*_{k,i}$ has the following closure property. We omit the proof, which is straightforward.

	\begin{claim} \label{lowerboundclaim}
		Suppose $\omega < k < i$, with $k$ a successor ordinal and $i$ a limit ordinal. In $V^{\bb{P}_k}$, suppose $A \subseteq C^*_{k,i}$ is a directed set of size $<\kappa_k$. For each $\ell \in X^i_k$, let $\gamma_\ell = \sup(\{\gamma_{p(\ell)_1} \mid (p,c) \in A\})$ and suppose that, for all $\ell \in X^i_k$, $\Vdash_{\bb{P}_{k,\ell}} ``\gamma_\ell$ is approachable with respect to $\dot{\vec{a}}_\ell."$ Then $A$ has a lower bound in $C^*_{k,i}$.
	\end{claim}

	Lemma \ref{clemma} has the following immediate corollary.

	\begin{corollary}
		Let $k < i$ be ordinals, with $k$ a successor and $i$ a limit.
		\begin{enumerate}
			\item{In $V^{\bb{P}_k}$, $\bb{C}^*_{k,i}$ is a dense, strongly $\kappa_k$-strategically closed subset of $\bb{C}_{k,i}$.}
			\item{In $V^{\bb{P}_i}$, $\bb{P}_{i, i+1}$ is $\kappa_{i+1}$-distributive.}
		\end{enumerate}
	\end{corollary}

	\begin{proof}
		(1) is immediate. (2) follows from the Lemma \ref{clemma} together with the observation that, if $\kappa$ is a singular cardinal and a poset $\bb{P}$ is $\kappa$-distributive, then it is also $\kappa^+$-distributive.
	\end{proof}

	The fact that, for all $i < k$, $\bb{P}_{i,k}$ is $\kappa_i$-distributive in $V^{\bb{P}_i}$ means that $V^{\bb{P}} = \bigcup_{i \in \mathrm{On}} V^{\bb{P}_i}$ is a model of ZFC. It is also easy to see that, for all $i \in \mathrm{On}$, $\kappa_i = \aleph_i^{V^{\bb{P}}}.$ If $i > \omega$ is a limit ordinal, then, in $V^{\bb{P}_i * \dot{\bb{A}}_i * \dot{\bb{S}}_i}$, $S_i$ is a stationary subset of $S^{\kappa_{i+1}}_\omega$ that does not reflect at any ordinals in $S^{\kappa_{i+1}}_{>\kappa_\omega}$. Since $\bb{T}_i * \dot{\bb{P}}_{i+1, k}$ is countably-closed for all $k > i+1$, $S_i$ remains stationary in $V^{\bb{P}}$.

	It remains to show that, if $i > \omega$ is a limit ordinal, then $\mathrm{Refl}(\kappa_{i+1})$ holds in $V^{\bb{P}}$. Thus, fix a limit ordinal $i > \omega$. Since, for every $k > i+1$, $\bb{P}_{i+1, k}$ adds no new subsets of $\kappa_{i+1}$ (recall that $\bb{P}_{i+1,i+2}$ is trivial forcing if $i$ is a limit ordinal), it suffices to check that $\mathrm{Refl}(\kappa_{i+1})$ holds in $V^{\bb{P}_{i+1}}$.

	Let $G = G_{i+1}$ be $\bb{P}_{i+1}$-generic over $V$. For $k < i+1$, let $G_k$ be the $\bb{P}_k$-generic filter induced by $G$. If $\omega < k < i+1$ and $k$ is a limit ordinal, let $G_{k,0}$ be the $\bb{P}_k * \dot{\bb{A}}_k$-generic induced by $G$, and let $G_{k,1}$ be the $\bb{P}_k * \dot{\bb{A}}_k * \dot{\bb{S}}_k$-generic induced by $G$. For $k < k' \leq i+1$, let $G_{k,k'}$ be the $\bb{P}_{k,k'}$-generic filter over $V[G_k]$ induced by $G$. Let $T \in V[G]$ be a stationary subset of $\kappa_{i+1}$. By shrinking $T$ if necessary, we may assume that there is $k < i$ such that $T \subseteq S^{\kappa_{i+1}}_{\kappa_{k}}.$

	We first assume that $k < \omega$. Let $k^* = k+2$. In $V[G_{k+1}]$, $\kappa_{k^*}$ is still supercompact, so fix $j:V[G_{k+1}] \rightarrow M[G_{k+1}]$ witnessing that $\kappa_{k^*}$ is $\kappa_{i+1}$-supercompact. In $V[G_{k+1}]$, $j(\bb{Q}_{k+1}) = \mathrm{Coll}(\kappa_{k+1}, <j(\kappa_{k^*})$, and $\bb{P}_{k^*, i+1}$ is strongly $\kappa_{k+1}$-strategically closed, so, by Fact \ref{lift}, $j(\bb{Q}_{k+1}) \cong \bb{P}_{k+1, i+1} * \dot{\bb{R}}$, where $\dot{\bb{R}}$ is forced to be $\kappa_{k+1}$-closed. Thus, letting $H$ be $\bb{R}$-generic over $V[G]$, we can extend $j$ to $j:V[G_{k^*}] \rightarrow M[G*H]$.

	To lift $j$ further to have domain $V[G]$, we define a condition $p^* \in j(\bb{P}_{k^*, i+1})$ such that $p^* \leq j(p)$ for all $p \in G_{k^*, i+1}$. We recursively define $p^*(\alpha)$ for $\alpha \in [k^*, j(i+1))$. Thus, suppose $\alpha \in [k^*, j(i+1))$ and we have defined $p^* \restriction [k^*, \alpha)$. If $\alpha = \omega$ or $\alpha$ is the successor of a limit ordinal, then $p^*(\alpha)$ is a name for the sole condition in the trivial forcing. If $\alpha$ is the successor of a successor ordinal, then the $\alpha$-th iterand in $j(\bb{P}_{k^*, i+1})$ is a Levy collapse that is $j(\kappa_{k^*})$-directed closed, so we can let $p^*(\alpha)$ be a name forced by $p^* \restriction [k^*, \alpha)$ to be a lower bound for $\{j(p)(\alpha) \mid p \in G_{k^*, i+1}\}$. If $\alpha$ is a limit ordinal, let $\gamma_\alpha = \sup(\{j(g)(\alpha) \mid g \in \prod_{\ell \leq i}\kappa_{\ell + 1} \cap V\})$. Let $p^*(\alpha)_0$ be a name forced by $p^* \restriction [k^*, \alpha)$ to be equal to $\{\gamma_\alpha\} \cup \bigcup_{p \in G_{k^*, i+1}} j(p)(\alpha)_0.$ This will be forced to be a condition provided that $\gamma_\alpha$ is forced to be approachable with respect to the entry corresponding to $\alpha$ in $j(\langle \vec{a}_\ell \mid \ell \in X^i_k \rangle)$. The argument showing that this is the case can be found in \cite{reflection}. The $\bb{S}$ and $\bb{T}$ parts of the $\alpha$-th iterand in $j(\bb{P}_{k^*, i+1})$ are forced to be $j(\kappa_{\omega+1})$-directed closed, so we can find $(p^*(\alpha)_1, p^*(\alpha)_2)$ that is forced by $p^* \restriction \alpha {^\frown} p^*(\alpha)_0$ to be a lower bound for $\{(j(p)(\alpha)_1, j(p)(\alpha)_2) \mid p \in G_{k^*, i+1}\}$.

	Let $I$ be $j(\bb{P}_{k^*, i+1})$-generic over $V[G*H]$ with $p^* \in I$, and lift $j$ to $j:V[G] \rightarrow M[G*H*I]$. By familiar arguments, we can now argue that, if $T$ does not reflect in $V[G]$, then $T$ is not stationary in $V[G*H*I]$, so there is a club $C$ in $\kappa_{i+1}$ with $C \in V[G*H*I]$ such that $C \cap T = \emptyset$. But $I$ is generic for $j(\kappa^*)$-distributive forcing (recall $j(\kappa^*) > \kappa_{i+1}$) and thus could not have added $C$, so $C \in V[G*H]$. $H$ is generic for $\kappa_{k+1}$-closed forcing and, in $V[G]$, $AP_{\kappa_i}$ holds and $T$ is a stationary subset of $S^{\kappa_{i+1}}_{<\kappa_{k+1}}$. Thus, by Fact \ref{apstat}, $T$ remains stationary in $V[G*H]$. This is a contradiction, so $T$ reflects in $V[G]$.

	Next, suppose $k > \omega$. Let $k < i^* < i$, with $i^*$ a successor ordinal large enough so that $\kappa_{i^*} > i$ and it is not the case that $\Vdash_{\bb{V}_{i^*+1, i}}``T$ is non-stationary.$"$ Let $k^* = i^*+1$, and fix $j:V[G_{i^*}] \rightarrow M[G_{i^*}]$ witnessing that $\kappa_{k^*}$ is $\kappa_{i+1}$-supercompact in $V[G_{i^*}]$. Recall that $\bb{C}_{k^*, i}$ has a dense, strongly $\kappa_{i^*}$-strategically-closed forcing. Thus, by Fact \ref{lift}, $j(\bb{Q}_{i^*}) \cong \bb{Q}_{i^*} * \dot{\bb{C}}_{k^*, i} * \dot{\bb{R}} \cong \bb{P}_{i^*, i+1} * \dot{\bb{V}}_{k^*, i} * \dot{\bb{R}}$, where $\dot{\bb{R}}$ is forced to be $\kappa_{i^*}$-closed. Let $H$ be $\bb{V}_{k^*, i}$-generic over $V[G]$ such that $T$ remains stationary in $V[G*H]$, let $I$ be $\bb{R}$-generic over $V[G*H]$, and extend $j$ to $j:V[G_{k^*}] \rightarrow M[G*H*I]$.

	We again define a condition $p^* \in j(\bb{P}_{k^*, i+1})$ such that $p^* \leq j(p)$ for all $p \in G_{k^*, i+1}$. Since $i < \kappa_{k^*}$, the domain of conditions in $j(\bb{P}_{k^*, i+1})$ is $[k^*, i+1)$. If $\ell \in [k^*, i+1)$ is a limit ordinal, let $\gamma_\ell = \sup(j``\kappa_{\ell+1})$. Note that $\gamma_\ell = \sup(\{\gamma_{p(\ell)_1} \mid (p,c) \in (G*H) \cap \bb{C}^*_{k^*, i}$ and that, as in the previous case and proven in \cite{reflection}, $\gamma_\ell$ is forced to be approachable with respect to $j(\dot{\vec{a}}_\ell)$. Thus, by Claim \ref{lowerboundclaim}, $\{j(p,c) \mid (p,c) \in (G*H) \cap \bb{C}^*_{k^*, i} \}$ has a lower bound. Let $(p^*, c^*)$ be such a lower bound and note that, since $\bb{C}^*_{k^*, i}$ is dense in $\bb{C}_{k^*, i}$, $p^* \leq j(p)$ for all $p \in G_{k^*, i+1}$.

	As in the previous case, let $J$ be $j(\bb{P}_{k^*, i+1})$-generic with $p^* \in J$ and lift $j$ to $j:V[G] \rightarrow M[G*H*I*J]$. As before, we argue that, if $T$ does not reflect in $V[G]$, then it is not stationary in $V[G*H*I*J]$. As before, we can pull the non-stationarity back to $V[G*H]$. However, we chose $H$ so that $T$ remains stationary in $V[G*H]$. This is a contradiction, so $T$ reflects in $V[G]$. 
\end{proof}

\section{Bounded stationary reflection without approachability} \label{approachSect1}

In the previous results, in order to obtain a model in which $\mu$ is a singular cardinal and $\mathrm{bRefl}(\mu^+)$ holds, we forced $AP_\mu$. In the final two sections of this paper, we produce models in which $\mathrm{bRefl}(\mu^+)$ holds and $AP_\mu$ fails. We first find such a model in which $\mu$ is a limit of large cardinals.

Let $\langle \kappa^0_n \mid n < \omega \rangle$ and $\langle \kappa^1_n \mid n < \omega \rangle$ be increasing sequences of supercompact cardinals such that, letting $\kappa^i_\omega = \sup(\{\kappa^i_n \mid n < \omega \})$ for $i \in \{0,1\}$, we have $\kappa^0_\omega < \kappa^1_0$. For $i \in \{0,1\}$, let $\lambda_i = (\kappa^i_\omega)^+$.

\begin{theorem}
	Assume GCH. There is a cardinal-preserving forcing extension in which:
	\begin{enumerate}
		\item{$\mathrm{Refl}(\lambda_1)$ holds.}
		\item{There is a stationary $S\subseteq S^{\lambda_1}_\omega$ that does not reflect at any ordinals in $S^{\lambda_1}_{\geq \lambda_0}$.}
		\item{$AP_{\kappa^1_\omega}$ fails.}
	\end{enumerate}
\end{theorem}

\begin{proof}
	By first forcing with Laver's preparatory forcing \cite{laver}, we may assume that, for all $i \in \{0,1\}$ and $n < \omega$, $\kappa^i_n$ remains supercompact in any forcing extension by a $\kappa^i_n$-directed closed forcing poset. Let $\bb{S} = \bb{S}^{\lambda_1}_{\omega, \lambda_0}$, and let $\dot{S}$ be a name for the stationary subset of $S^{\lambda_1}_\omega$ added by $\bb{S}$. In $V^{\bb{S}}$, let $\bb{T} = CU(S)$.

	In $V^{\bb{S}}$, define a forcing iteration $\langle \bb{P}_\xi, \dot{\bb{Q}}_\zeta \mid \xi \leq \lambda_1^+, \zeta < \lambda_1^+ \rangle$, taken with supports of size $\kappa^1_\omega$, as follows. If $\zeta < \lambda_1^+$ and $\bb{P}_\zeta$ has been defined, choose a $\bb{P}_\zeta$-name $\dot{T}_\zeta$ for a subset of $S^{\lambda_1}_{\geq \lambda_0}$ such that: $\Vdash_{\bb{P}_\zeta * \dot{\bb{T}}} ``\dot{T}_\zeta$ is not stationary," and let $\dot{\bb{Q}}_\zeta$ be a $\bb{P}_\zeta$-name forced to be equal to $CU(\dot{T}_\zeta)$. Let $\bb{P} = \bb{P}_{\lambda^+_1}$. By standard bookkeeping arguments, we can arrange so that, in $V^{\bb{S}*\dot{\bb{P}}}$, if $X \subseteq S^{\lambda_1}_{\geq \lambda_0}$ is such that $\Vdash_{\bb{T}}``X$ is nonstationary," then $X$ is already nonstationary in $V^{\bb{S}*\dot{\bb{P}}}$.

	Let $G$ be $\bb{S}$-generic over $V$, and let $H$ be $\bb{P}$-generic over $V[G]$. We claim that $V[G*H]$ is the desired model. By Lemma \ref{stat_dest}, $\bb{S} * \dot{\bb{P}} * \dot{\bb{T}}$ has a $\lambda_1$-closed dense subset (in fact, an examination of the proof shows that it actually has a $\lambda_1$-directed closed dense subset). Thus, $\bb{S} * \dot{\bb{P}}$ is $\lambda_1$-distributive. Also, $\bb{S} * \dot{\bb{P}}$ is easily seen to have the $\lambda_1^+$-c.c., so forcing with it preserves cardinals. Let $S$ be the set enumerated by $\bigcup G$. In $V[G]$, $S$ is a stationary subset of $S^{\lambda_1}_\omega$ that does not reflect at any ordinal in $S^{\lambda_1}_{\geq \lambda_0}$. Note that $\bb{P}$ is $\lambda_0$-directed closed, so $S$ remains stationary in $V[G*H]$ and still does not reflect at any ordinals in $S^{\lambda_1}_{\geq \lambda_0}$.

	We now verify that $\mathrm{Refl}(\lambda_1)$ holds in $V[G*H]$. To this end, let $T$ be a stationary subset of $\lambda_1$. Without loss of generality, by shrinking $T$ if necessary, we can assume that there is $\mu < \kappa^1_\omega$ such that $T \subseteq S^{\lambda_1}_\mu$. First, suppose $\mu < \kappa^0_\omega$. Let $n^* < \omega$ be such that $\mu < \kappa^0_{n^*}$. Since $\bb{S}*\dot{\bb{P}}$ is $\lambda_0$-directed closed in $V$, $\kappa^0_{n^*}$ remains supercompact in $V[G*H]$. Let $j:V[G*H] \rightarrow M$ witness that $\kappa^0_{n^*}$ is $\lambda_1$-supercompact. Let $\delta = \sup(j``\lambda_1)$. As in earlier arguments, $j(T)$ reflects at $\delta$ in $M$. By elementarity, $T$ reflects at some ordinal $\alpha < \lambda_1$ in $V[G*H]$.

	Next, suppose $\lambda_0 \leq \mu < \kappa^1_\omega$. Let $n^* < \omega$ be such that $\mu < \kappa^1_{n^*}$. As $T \subseteq S^{\lambda_1}_{\geq \lambda_0}$ is stationary in $V[G*H]$, it is not the case that $\Vdash_{\bb{T}}``T$ is non-stationary.$"$ Thus, let $I$ be $\bb{T}$-generic over $V[G*H]$ such that $T$ remains stationary in $V[G*H*I]$. Since $\bb{S} * \dot{\bb{P}} * \dot{\bb{T}}$ has a $\lambda_1$-directed closed dense subset, $\kappa^1_{n^*}$ is supercompact in $V[G*H*I]$. Let $j:V[G*H*I] \rightarrow M$ witness that $\kappa^1_{n^*}$ is $\lambda_1$-supercompact. By the same arguments as in the previous case, $T$ reflects at some ordinal $\alpha < \lambda_1$ in $V[G*H*I]$. Since this statement is obviously downward absolute, it reflects in $V[G*H]$.

	It remains to show that $AP_{\kappa^1_\omega}$ fails in $V[G*H]$. However, this follows from the fact that $\kappa^0_0$ is supercompact in $V[G*H]$ and the fact that, by a result of Shelah, if $\cf(\mu) < \kappa < \mu$ and $\kappa$ is supercompact, then $AP_\mu$ fails (see \cite{cummings} for a proof).
\end{proof}

\section{Down to smaller cardinals} \label{approachSect2}

We would like to bring the results of the previous section down to smaller cardinals. By the following result of Chayut \cite{chayut}, assuming some cardinal arithmetic, $\aleph_{\omega^2 + 1}$ is the smallest we can hope for.

\begin{theorem}
	Suppose $n < \omega$, $\aleph_{\omega \cdot n}$ is strong limit, $2^{\aleph_{\omega \cdot n}} = \aleph_{\omega \cdot n + 1}$, and $\mathrm{Refl}(\aleph_{\omega \cdot n + 1})$ holds. Then $AP_{\aleph_{\omega \cdot n}}$ holds.
\end{theorem}

We do not succeed in bringing the result from Section \ref{approachSect1} down to $\aleph_{\omega^2 + 1}$, but we can attain it at $\aleph_{\omega^2 \cdot 2 + 1}$. In this section we adopt the convention, for notational simplicity, that if we are working in a forcing extension $V^{\bb{P}}$ of $V$, then $G(\bb{P})$ denotes the $\bb{P}$-generic filter over $V$ used to define the extension.

\begin{theorem}
	Suppose there is an increasing sequence of supercompact cardinals of order type $\omega \cdot 2$. Then there is a forcing extension in which $\mathrm{Refl}(\aleph_{\omega^2 \cdot 2 + 1})$ holds, $AP_{\aleph_{\omega^2 \cdot 2}}$ fails, and there is a stationary subset of $S^{\aleph_{\omega^2 \cdot 2 + 1}}_\omega$ that does not reflect at any ordinals in $S^{\aleph_{\omega^2 \cdot 2 + 1}}_{\aleph_{\omega^2 + 1}}$.
\end{theorem}

\begin{proof}
	We follow to a large extent Section 3 of \cite{magidorshelah}. Assume GCH. Let $\langle \kappa^0_n \mid n < \omega \rangle$ and $\langle \kappa^1_n \mid n < \omega \rangle$ be two increasing sequences of supercompact cardinals such that, for all $n < \omega$, $\kappa^0_n < \kappa^1_0$. Assume that each $\kappa^i_n$ is indestructible under $\kappa^i_n$-directed closed forcing. For $i < 2$, let $\kappa^i_\omega = \sup(\{\kappa^i_n \mid n < \omega \})$, and let $\lambda_i = (\kappa^i_\omega)^+$. For notational simplicity, let $\kappa^0_{-1} = \omega_1$ and $\kappa^1_{-1} = \lambda_0$. For $i < 2$ and $n < \omega$, let $\bb{C}^i_n = \prod_{m \geq n} \mathrm{Coll}((\kappa^i_{m-1})^{++}, < \kappa^i_m)$, where the product is taken with full support.

	Define an equivalence relation on $\bb{C}^1_0$ by declaring that $c_0 \equiv c_1$, where $c_i = \langle c_i(n) \mid n < \omega \rangle$, if $c_0(n) = c_1(n)$ for all but finitely many $n < \omega$. Let $\bb{C}^*$ be the forcing notion whose conditions are equivalence classes from $\bb{C}^1_0$ and such that $[c_1] \leq [c_0]$ if, for all but finitely many $n < \omega$, $c_1(n) \leq c_0(n)$. The following is proven in \cite{magidorshelah}.

	\begin{proposition}
		For all $n < \omega$, there is a projection from $\bb{C}^1_n$ onto $\bb{C}^*$. Hence, $\bb{C}^*$ is $\lambda_1$-distributive. Moreover, if $G$ is $\bb{C}^*$-generic over $V$ and $n < \omega$, then $\bb{C}^1_n/G$ has the $\lambda_1$-c.c.
	\end{proposition}

	In $V^{\bb{C}^*}$, let $\bb{S} = \bb{S}^{\lambda_1}_{\omega, \lambda_0}$. Let $\dot{S}$ be an $\bb{S}$-name for the stationary subset of $S^{\lambda_1}_\omega$ added by $\bb{S}$ and, in $V^{\bb{C}^* * \dot{\bb{S}}}$, let $\bb{T} = CU(\dot{S})$. Also in $V^{\bb{C}^* * \dot{\bb{S}}}$, let $\bb{Q} = \bb{Q}_{\lambda^+}$ be an iteration of length $\lambda_1^+$ with supports of size $\kappa^1_\omega$ of forcings to destroy certain stationary subsets of $S^{\lambda_1}_{\geq \lambda_0}$. As in Section \ref{destroySection}, we can arrange so that, in $V^{\bb{C^*}}$, $\bb{S} * \dot{\bb{Q}} * \dot{\bb{T}}$ has a dense $\lambda_1$-closed subset and if, in $V^{\bb{C}^* * \dot{\bb{S}} * \dot{\bb{Q}}}$, $T \subseteq S^{\lambda_1}_{\geq \lambda_0}$ and $\Vdash_{\bb{T}}``T$ is non-stationary," then $T$ is already non-stationary in $V^{\bb{C}^* * \dot{\bb{S}} * \dot{\bb{Q}}}$. Since $\bb{T}$ is weakly homogeneous, we in fact get that, for all $T \subseteq S^{\lambda_1}_{\geq \lambda_0} \cap V^{\bb{C}^* * \dot{\bb{S}} * \dot{\bb{Q}}}$, if $T$ is stationary, then $\Vdash_{\bb{T}}``T$ is stationary."

	For $n < \omega$, note that $\bb{C}^0_{n+1} \times (\bb{C}^1_0 * \dot{\bb{S}} * \dot{\bb{Q}})$ is $(\kappa^0_n)^{++}$-directed closed. Let $\dot{F}^0_n$ denote a name for a fine, normal ultrafilter on $\mathcal{P}_{\kappa^0_n}(\lambda_1)$ in $V^{\bb{C}^0_{n+1} \times (\bb{C}^1_0 * \dot{\bb{S}} * \dot{\bb{Q}})}$. Let $U^0_n$ be its projection to a normal ultrafilter on $\kappa^0_n$. Note that $U^0_n \in V$ and, by the homogeneity of the forcing, we may assume that the trivial condition forces $U^0_n$ to be the projection of a fine, normal ultrafilter on $\mathcal{P}_{\kappa^0_n}(\lambda_1)$. Similarly, define a normal ultrafilter $U^1_n$ on $\kappa^1_n$ such that the trivial condition in $\bb{C}^1_{n+1} * \dot{\bb{S}} * \dot{\bb{Q}} * \dot{\bb{T}}$ forces $U^1_n$ to be the projection of a fine, normal ultrafilter on $\mathcal{P}_{\kappa^1_n}(\lambda_1)$.

	For $i < 2$ and $n < \omega$, let $M^i_n$ denote the transitive collapse of $\mathrm{Ult}(V, U^i_n)$, and let $j^i_n: V \rightarrow M^i_n$ be the associated embedding. Let $\bb{T}^0_n$ denote $\mathrm{Coll}((\kappa^0_n)^{+\omega \cdot 2 + 2}, < j^0_n(\kappa^0_n))$ as defined in $M^0_n$. $M \models ``$there are $j^0_n(\kappa^0_n)$ maximal antichains of $\bb{T}^0_n$." Since $|j^0_n(\kappa^0_n)| = (\kappa^0_n)^+$ and $\bb{T}^0_n$ is $(\kappa^0_n)^+$-closed, we can build in $V$ a $\bb{T}^0_n$-generic filter over $M^0_n$. Let $G^0_n$ be such a filter. Similarly, let $\bb{T}^1_n$ denote $\mathrm{Coll}((\kappa^1_n)^{+\omega + 2}, < j^1_n(\kappa^1_n))$ as defined in $M^1_n$, and fix $G^1_n$, a $\bb{T}^1_n$-generic filter over $M^1_n$.

	We now define diagonal Prikry forcing notions $\bb{P}_0$ and $\bb{P}_1$, which are slightly modified versions of the forcing in \cite{magidorshelah}. Elements of $\bb{P}_0$ are of the form $p = \langle \alpha^p_0,\ldots,\alpha^p_{n-1}, \langle A^p_k \mid n \leq k < \omega \rangle, g^p_0,\ldots,g^p_n,f^p_0,\ldots,f^p_{n-1},\langle F^p_k \mid n \leq k < \omega \rangle,\langle g^p_k \mid n < k < \omega \rangle \rangle,$ where
	\begin{itemize}
		\item{For all $i < n$, $\alpha^p_i$ is inaccessible and $\kappa^0_{i-1} < \alpha^p_i < \kappa^0_i$.}
		\item{For all $n \leq k < \omega$, $A^p_k \in U^0_k$ and, for all $\alpha \in A^p_k$, $\alpha$ is inaccessible.}
		\item{For all $i < n$, $g^p_i \in \mathrm{Coll}((\kappa^0_{i-1})^{++}, < \alpha^p_i)$ and $f^p_i \in \mathrm{Coll}((\alpha^p_i)^{+\omega \cdot 2 + 2}, < \kappa^0_i)$.}
		\item{For all $n \leq k < \omega$, $g^p_k \in \mathrm{Coll}((\kappa^0_{k-1})^{++}, < \kappa^0_k)$ is such that, for all $\alpha \in A^p_k$, $g^p_k \in \mathrm{Coll}((\kappa^0_{k+1})^{++}, < \alpha)$.}
		\item{For all $n \leq k < \omega$, $F^p_k$ is a function with domain $A^p_k$ such that, for all $\alpha \in A^p_k$, $F^p_k(\alpha) \in \mathrm{Coll}(\alpha^{+\omega \cdot 2 + 2}, < \kappa^0_k)$ and $j^0_k(F^p_k)(\kappa^0_k) \in G^0_k$.}
	\end{itemize}

	$n$ is the length of $p$ and is denoted $\ell(p)$. If $q, p \in \bb{P}_0$, then $q \leq p$ if:
	\begin{itemize}
		\item{$\ell(q) \geq \ell(p)$.}
		\item{For all $i < \ell(p)$, $\alpha^q_i = \alpha^p_i$ and $f^q_i \leq f^p_i$.}
		\item{For all $i < \omega$, $g^q_i \leq g^p_i$.}
		\item{For all $\ell(q) \leq k < \omega$, $A^q_k \subseteq A^p_k$ and, for all $\alpha \in A^q_k$, $F^q_k(\alpha) \leq F^p_k(\alpha)$.}
		\item{For all $\ell(p) \leq k < \ell(q)$, $\alpha^q_k \in A^p_k$ and $f^q_k \leq F^p_k(\alpha^q_k)$.}
	\end{itemize}

	$\bb{P}_1$ is defined in the same way, with the following changes:
	\begin{itemize}
		\item{For all $-1 \leq i < \omega$, every occurrence of $\kappa^0_i$ in the definition of $\bb{P}_0$ is replaced by $\kappa^1_i$ in the definition of $\bb{P}_1$ and every occurrence of $U^0_i$ is replaced by $U^1_i$.}
		\item{If $p \in \bb{P}_1$ and $i < \ell(p)$, then $f^p_i \in \mathrm{Coll}((\alpha^p_i)^{+\omega + 2}, < \kappa^1_i)$. If $\ell(p) \leq k < \omega$ and $\alpha \in A^p_k$, then $F^p_k(\alpha) \in \mathrm{Coll}(\alpha^{+\omega+2}, < \kappa^1_j)$ and $j^1_k(F^p_k)(\kappa^1_k) \in G^1_k$.}
	\end{itemize}

	Following \cite{magidorshelah}, if $p \in \bb{P}_i$, we call $\langle \alpha^p_k \mid k < \ell(p) \rangle$ its $\alpha$-part, $\langle A^p_k \mid \ell(p) \leq k < \omega \rangle$ its $A$-part, $\langle f^p_k \mid k < \ell(p) \rangle$ its $f$-part, $\langle g^p_k \mid k \leq \ell(p) \rangle$ its $g$-part, $\langle F^p_k \mid \ell(p) \leq k < \omega \rangle$ its $F$-part, and $\langle g^p_k \mid \ell(p) < k < \omega \rangle$ its $C$-part. The $\alpha$-part, $g$-part, and $f$-part together comprise the lower part of $p$, denoted $a(p)$. If $k \leq \ell(p)$, let $p \restriction k$ denote $\langle \langle \alpha^p_m \mid m < k \rangle, \langle g^p_m \mid m \leq k \rangle, \langle f^p_m \mid m < k \rangle \rangle$. Note that $p \restriction \ell(p) = a(p)$. If $q,p \in \bb{P}_i$, then we say $q$ is a length-preserving extension of $p$ if $q \leq p$ and $\ell(q) = \ell(p)$. If $k \leq \ell(p)$, then $q$ is a $k$-length-preserving extension of $p$ if $q$ is a length-preserving extension of $p$ and $q \restriction k = p \restriction k$. We say $q$ is a trivial extension of $p$ if it is an $\ell(p)$-length-preserving extension of $p$.

	$\bb{P}_i$ satisfies a form of the Prikry lemma. A proof can be found in \cite{magidorshelah}.

	\begin{lemma}
		Let $p \in \bb{P}_i$, let $k \leq \ell(p)$, and let $D$ be a dense open subset of $\bb{P}_i$. Then there is a $k$-length-preserving extension $q \leq p$ such that, if $q^* \leq q$ and $q^* \in D$, then, if $q^{**} \leq q$, $\ell(q^{**}) = \ell(q^*)$, and $q^{**} \restriction k = q^* \restriction k$, then $q^{**} \in D$.
	\end{lemma}

	The Prikry lemma can be applied to see that the only cardinals below $\lambda_i$ that are collapsed by forcing with $\bb{P}_i$ are those explicitly in the scope of the Levy collapses interleaved into the forcing notion. In particular, the following claim immediately follows.

	\begin{claim} \label{unbddSubset}
		Let $k < \omega$. In $V^{\bb{P}_0}$, if $\cf(\alpha) = (\alpha^p_k)^{+\omega \cdot 2 + 1}$ (equivalently, $\cf^V(\alpha) = (\alpha^p_k)^{+\omega \cdot 2 + 1}$) for some $p \in G(\bb{P}_0)$ and $A$ is an unbounded subset of $\alpha$, then there is an unbounded $B \subseteq A$ such that $B \in V$. Similarly, in $V^{\bb{P}_1}$, if $\cf(\alpha) = (\alpha^p_k)^{+\omega + 1}$ for some $p \in G(\bb{P}_1)$ and $A$ is an unbounded subset of $\alpha$, then there is an unbounded $B \subseteq A$ such that $B \in V$.
	\end{claim}

	Other arguments, found in \cite{magidorshelah}, imply that all cardinals $\geq \lambda_i$ are preserved as well. Note that, by the Prikry lemma, forcing with $\bb{P}_1$ does not add any new bounded subsets of $\lambda_0^{++}$, so $\bb{P}_0$ has the same basic properties in $V^{\bb{P}_1}$ as it has in $V$.

	The following is proven in \cite{magidorshelah}.

	\begin{proposition}
		There is a projection from $\bb{P}_1$ onto $\bb{C}^*$ such that the quotient forcing has the $\lambda_1$-c.c.
	\end{proposition}

	Our final model will be $V^{\bb{P}_0 \times (\bb{P}_1 * \dot{\bb{S}} * \dot{\bb{Q}})}$. Let $S$ be the subset of $S^{\lambda_1}_\omega$ added by $\bb{S}$. In $V^{\bb{C}^* * \dot{\bb{S}}}$, $S$ is stationary and does not reflect at any ordinals in $S^{\lambda_1}_{\geq \lambda_0}$. In $V^{\bb{C}^* * \dot{\bb{S}}}$, $\bb{Q}$ is $\lambda_0$-closed, so $S$ remains stationary in $V^{\bb{C}^* * \dot{\bb{S}} * \dot{\bb{Q}}}$. In $V^{\bb{C}^*}$, $\bb{P}_1/G(\bb{C}^*)$ has the $\lambda_1$-c.c. $\bb{S} * \dot{\bb{Q}}$ is the projection of a forcing poset with a dense $\lambda_1$-closed subset (namely $\bb{S} * \dot{\bb{Q}} * \dot{\bb{T}}$), so, by Easton's Lemma, $\bb{P}_1/G(\bb{C}^*)$ has the $\lambda_1$-c.c. in $V^{\bb{C}^* * \dot{\bb{S}} * \dot{\bb{Q}}}$, so $S$ is stationary in $V^{\bb{P}_1 * \dot{\bb{S}} * \dot{\bb{Q}}}$. Finally, $|\bb{P}_0| = \lambda_0$, so $S$ remains stationary in $V^{\bb{P}_0 \times (\bb{P}_1 * \dot{\bb{S}} * \dot{\bb{Q}})}$.

	We now verify that every stationary subset of $\lambda_1$ reflects in $V^{\bb{P}_0 \times (\bb{P}_1 * \dot{\bb{S}} * \dot{\bb{Q}})}$. Thus, let $T \in V^{\bb{P}_0 \times (\bb{P}_1 * \dot{\bb{S}} * \dot{\bb{Q}})}$ be a stationary subset of $\lambda_1$, and let $\dot{T}$ be a name for it. Let $(i,n)$ be the lexicographically least pair such that $T \cap S^{\lambda_1}_{<\kappa^i_n}$ is stationary. By shrinking $T$ if necessary, we may assume that, for some $(p_0,(p_1,\dot{s},\dot{q})) \in G(\bb{P}_0 \times (\bb{P}_1 * \dot{\bb{S}} * \dot{\bb{Q}}))$, $(p_0, (p_1,\dot{s},\dot{q}))$ forces that $\dot{T}$ is a stationary subset of $S^{\lambda_1}_{<\kappa^i_n}$. Moreover, if $i = 1$, we may assume that $(p_0, (p_1,\dot{s},\dot{q}))$ forces $\dot{T}$ to be a stationary subset of $S^{\lambda_1}_{\geq \lambda_0}$ as well. For each $\alpha \in T$, let $(p^\alpha_0, (p^\alpha_1,\dot{s}^\alpha,\dot{q}^\alpha)) \in G(\bb{P}_0 \times (\bb{P}_1 * \dot{\bb{S}} * \dot{\bb{Q}}))$, with $(p^\alpha_0, (p^\alpha_1,\dot{s}^\alpha,\dot{q}^\alpha)) \leq (p_0, (p_1,\dot{s},\dot{q}))$, force that $\alpha \in \dot{T}$. Since $|\bb{P}_0| = \lambda_0$ and there are only $\kappa^1_\omega$ lower parts in $\bb{P}_1$, we may assume there are $p^*_0 \in \bb{P}_0$ and $a^*$, a lower part for $\bb{P}_1$, such that, for every $\alpha \in T$, $p^\alpha_0 = p^*_0$ and $a(p^\alpha_1) = a^*$. We may also assume that $p_0 = p^*_0$, $a(p_1) = a^*$, and $(p_0, (p_1,\dot{s},\dot{q}))$ forces that $T^* := \{\alpha \mid$ for some $(p'_0, (p'_1,\dot{s}',\dot{q}')) \in G(\bb{P}_0 \times (\bb{P}_1 * \dot{\bb{S}} * \dot{\bb{Q}}))$ such that $p'_0 = p_0$, $p'_1$ is a trivial extension of $p_1$, and $p'_1 \Vdash ``(\dot{s}', \dot{q}') \leq (\dot{s}, \dot{q})"$, $(p'_0, (p'_1,\dot{s}',\dot{q}'))$ forces that $\alpha \in \dot{T}\}$ is stationary. Let $\dot{T}^*$ be a name for $T^*$. Finally, we may assume that $\ell(p_i) \geq n$. We will find an extension of $(p_0, (p_1, \dot{s}, \dot{q}))$ forcing that $\dot{T}$ reflects. There are two cases to consider.

	\textbf{Case 1: i = 0.} Let $n^* = \ell(p_0)$. Move to $V^{\bb{C}^0_{n^*+1} \times (\bb{C}^1_0 * \dot{\bb{S}} * \dot{\bb{Q}})}$, requiring that the $C$-part of $p_0$ is in $G(\bb{C}^0_{n^*+1})$ and $(\langle g^{p_1}_k \mid k < \omega \rangle, \dot{s}, \dot{q}) \in G(\bb{C}^1_0 * \dot{\bb{S}} * \dot{\bb{Q}})$. For $n^* + 1 \leq m < \omega$, let $G(\bb{C}^0_m)$ be the generic filter induced by $G(\bb{C}^0_{n^*+1})$. Similarly, for $m < \omega$, let $G(\bb{C}^1_m)$ be the generic filter induced by $G(\bb{C}^1_0)$. Let $\bb{P}^*$ be the set of $(r_0, r_1) \in \bb{P}_0 \times \bb{P}_1$ such that $(r_0, r_1) \leq (p_0, p_1)$, the $C$-part of $r_0$ is in $G(\bb{C}^0_{\ell(r_0) + 1})$, and the $C$-part of $r_1$ is in $G(\bb{C}^1_{\ell(r_1) + 1})$. The proof of Lemma 6 from Section 3 of \cite{magidorshelah} shows that forcing with $\bb{P}^*$ over $V^{\bb{C}^0_{n^*+1} \times (\bb{C}^1_0 * \dot{\bb{S}} * \dot{\bb{Q}})}$ adds a $V$-generic filter for $\bb{P}_0 \times \bb{P}_1$. In $V^{\bb{C}^0_{n^*+1} \times (\bb{C}^1_0 * \dot{\bb{S}} * \dot{\bb{Q}})}$, let $\hat{T}$ be the set of $\alpha < \lambda_1$ such that, for some $r_1$ such that $r_1$ is a trivial extension of $p_1$ and $(p_0, r_1) \in \bb{P}^*$, and, for some $(\dot{r}_2, \dot{r}_3)$ such that $(p_0, r_1) \Vdash ``(\dot{r}_2, \dot{r}_3) \in G(\dot{\bb{S}} * \dot{\bb{Q}})$ and $(\dot{r}_2, \dot{r}_3) \leq (\dot{s}, \dot{q})"$, we have $(p_0, (r_1, \dot{r}_2, \dot{r}_3)) \Vdash ``\alpha \in \dot{T^*}"$

	\begin{lemma}
		$\hat{T}$ is stationary.
	\end{lemma}

	\begin{proof}
		Suppose not. Then, in $V^{\bb{C}^0_{n^*+1} \times (\bb{C}^1_0 * \dot{\bb{S}} * \dot{\bb{Q}})}$, there is a club $C$ in $\lambda_1$ such that $C \cap \hat{T} = \emptyset$. Force with $\bb{P^*}$. in $V^{(\bb{C}^0_{n^*+1} \times (\bb{C}^1_0 * \dot{\bb{S}} * \dot{\bb{Q}})) * \bb{P^*}}$, $\hat{T} \supseteq T^*$. Thus, $C \cap T^* = \emptyset$. $V^{\bb{C}^0_{n^*+1} \times (\bb{C}^1_0 * \dot{\bb{S}} * \dot{\bb{Q}})}$ is a forcing extension of $V^{(\bb{C}^* * \dot{\bb{S}} * \dot{\bb{Q}})}$ by a $\lambda_1$-c.c. forcing extension, so there is a club $D \subseteq C$ such that $D \in V^{(\bb{C}^* * \dot{\bb{S}} * \dot{\bb{Q}})}$. But $V^{(\bb{C}^* * \dot{\bb{S}} * \dot{\bb{Q}})} \subseteq V^{\bb{P}_0 \times (\bb{P}_1 * \dot{\bb{S}} * \dot{\bb{Q}})}$ and $T^*$ is stationary in $V^{\bb{P}_0 \times (\bb{P}_1 * \dot{\bb{S}} * \dot{\bb{Q}})}$. This contradicts the fact that $D \in V^{\bb{P}_0 \times (\bb{P}_1 * \dot{\bb{S}} * \dot{\bb{Q}})}$ is club in $\lambda_1$ and disjoint from $T^*$.
	\end{proof}

	In $V^{\bb{C}^0_{n^*+1} \times (\bb{C}^1_0 * \dot{\bb{S}} * \dot{\bb{Q}})}$, $\kappa^0_{n^*}$ remains supercompact, and $\lambda_1 = (\kappa^0_{n^*})^{+\omega \cdot 2 + 1}$. Fix a fine, normal measure $U^*$ on $\mathcal{P}_{\kappa^0_{n^*}}(\lambda_1)$ such that $U^*$ projects to $U^0_{n^*}$. Let $\theta$ be a sufficiently large regular cardinal, and let $\mathfrak{A}$ denote an expansion of $(H(\theta), \in)$ by a well-ordering of $H(\theta)$ and constants for all relevant sets. The following are standard applications of supercompactness.

	\begin{lemma}
		Let $E_0 = \{X \in \mathcal{P}_{\kappa^0_{n^*}}(\lambda_1) \mid $for some $\mathfrak{B} \prec \mathfrak{A}$, we have $X = \mathfrak{B} \cap \lambda_1$, $|X| = |\mathfrak{B}|$, and $X \in \bigcap_{A \in U^* \cap \mathfrak{B}} A\}$. Then $E_0 \in U^*$.
	\end{lemma}

	\begin{lemma}
		Let $E_1 = \{X \in \mathcal{P}_{\kappa^0_{n^*}}(\lambda_1) \mid X \cap \kappa^0_{n^*}$ is inaccessible, $\mathrm{otp}(X) = (X \cap \kappa^0_{n^*})^{+\omega \cdot 2 + 1}$, and $\hat{T} \cap X$ is stationary in $\sup(X)\}$. Then $E_1 \in U^*$.
	\end{lemma}

	The next lemma follows from the proof of Lemma 13 in \cite{magidorshelah}.

	\begin{lemma} \label{absorbLemma}
		Let $X \in E_0 \cap E_1$ such that $X \cap \kappa^0_{n^*} \in A^{p_0}_{n^*}$. Let $\mathfrak{B} \prec \mathfrak{A}$ witness that $X \in E_0$. Then there is $(p^*_0, p^*_1) \in \bb{P}^*$ such that:
		\begin{enumerate}
			\item{$(p^*_0,p^*_1) \leq (p_0,p_1)$.}
			\item{$\ell(p^*_0) = n^* + 1$ and $p^*_1$ is a trivial extension of $p_1$.}
			\item{$\alpha^{p^*_0}_{n^*} = X \cap \kappa^0_{n^*}$.}
			\item{If $(p_0,p'_1) \in \bb{P}^* \cap \mathfrak{B}$ and $p'_1$ is a trivial extension of $p_1$, then $(p^*_0,p^*_1) \leq (p_0,q'_1)$.}
		\end{enumerate}
	\end{lemma}

	Let $X, \mathfrak{B}$, and $(p^*_0,p^*_1)$ be as given in Lemma \ref{absorbLemma}. For every $\gamma \in \hat{T} \cap X$, there is $p^1_\gamma$, a trivial extension of $p_1$, and $(\dot{s}_\gamma',\dot{q}_\gamma')$ such that $(p_0, p^1_\gamma)$ forces that $(\dot{s}_\gamma',\dot{q}_\gamma') \in G(\dot{\bb{S}}*\dot{\bb{Q}})$ and $(p_0, (p^1_\gamma, \dot{s}_\gamma', \dot{q}_\gamma'))$ forces that $\gamma \in \dot{T}^*$. By elementarity of $\mathfrak{B}$, such a $p^1_\gamma$ exists in $\mathfrak{B}$, and hence, for all $\gamma \in \hat{T} \cap X$, $(p^*_0, (p^*_1, \dot{s}'_\gamma, \dot{q}'_\gamma)) \Vdash ``\gamma \in \dot{T}^*"$. Since $\hat{T} \cap X \in V^{\bb{C}^0_{n^*+1} \times (\bb{C}^1_0 * \dot{\bb{S}} * \dot{\bb{Q}})}$ and has size less than $\kappa^0_{n^*}$, we have $\{(\dot{s}'_\gamma, \dot{q}'_\gamma) \mid \gamma \in \hat{T} \cap X\} \in V$. In $V^{\bb{C}^*}$, $\bb{S} * \dot{\bb{Q}}$ is $\kappa^0_{n^*}$-directed closed. Thus, we can find names $\dot{s}'$ and $\dot{q}'$ such that, for all $\gamma \in \hat{T} \cap X$, $(p^*_0, (p^*_1, \dot{s}', \dot{q}')) \leq (p^*_0, (p^*_1, \dot{s}'_\gamma, \dot{q}'_\gamma))$. Thus, $(p^*_0, (p^*_1, \dot{s}', \dot{q}')) \Vdash ``\dot{T} \cap X \supseteq \hat{T} \cap X"$. Since no cardinals between $X \cap \kappa^0_{n^*}$ and $(X \cap \kappa^0_{n^*})^{+ \omega \cdot 2 + 2}$ are collapsed, an application of Claim \ref{unbddSubset} yields that $\hat{T} \cap X$ remains stationary in $\mathrm{sup}(X)$ after forcing over $V^{\bb{C}^0_{n^*+1} \times (\bb{C}^1_0 * \dot{\bb{S}} * \dot{\bb{Q}})}$ with $\bb{P}^*$ below $(p^*_0,p^*_1)$. Hence, $T \cap X$ is stationary in $\sup(X)$ in $V^{\bb{P}_0 \times (\bb{P}_1 * \dot{\bb{S}} * \dot{\bb{Q}})}$ after forcing below $(q_0, (q_1, \dot{s}', \dot{q}'))$, so $(q_0, (q_1, \dot{s}', \dot{q}')) \Vdash``\dot{T}$ reflects."

	\textbf{Case 2: i = 1.} Let $n^* = \ell(p_1)$. Move to $V^{\bb{C}^1_{n^*+1} * \dot{\bb{S}} * \dot{\bb{Q}}}$, requiring that the $C$-part of $p_1$ is in $G(\bb{C}^1_{n^*+1})$. Let $\bb{P}^*$ be the set of $(r_0,r_1) \in \bb{P}_0 \times \bb{P}_1$ such that $(r_0,r_1) \leq (p_0,p_1)$ and the $C$-part of $r_1$ is in $G(\bb{C}^1_{\ell(r_1)+1})$. As before, forcing with $\bb{P}^*$ over $V^{\bb{C}^1_{n^*+1} * \dot{\bb{S}} * \dot{\bb{Q}}}$ adds a $V$-generic filter for $\bb{P}_0 \times \bb{P}_1$. Let $\hat{T}$ be the set of $\alpha < \lambda_1$ such that, for some $r_1$ such that $r_1$ is a trivial extension of $p_1$ and $(p_0, r_1) \in \bb{P}^*$, and for some $(\dot{r}_2, \dot{r}_3)$ such that $(p_0, r_1) \Vdash ``(\dot{r}_2, \dot{r}_3) \in G(\dot{S} * \dot{Q})$ and $(\dot{r}_2, \dot{r}_3) \leq (\dot{s}, \dot{q})"$, we have $(p_0, (r_1, \dot{r}_2, \dot{r}_3)) \Vdash ``\alpha \in \dot{T^*}"$. As in Case 1, $\hat{T}$ is stationary in $V^{\bb{C}^1_{n^*+1} * \dot{\bb{S}} * \dot{\bb{Q}}}$.

	\begin{lemma}
		$\hat{T}$ is stationary in $V^{\bb{C}^1_{n^*+1} * \dot{\bb{S}} * \dot{\bb{Q}} * \dot{\bb{T}}}$.
	\end{lemma}

	\begin{proof}
		Suppose not, and let $D \in V^{\bb{C}^1_{n^*+1} * \dot{\bb{S}} * \dot{\bb{Q}} * \dot{\bb{T}}}$ be club in $\lambda_1$ such that $D \cap \hat{T} = \emptyset$. Since $V^{\bb{C}^1_{n^*+1} * \dot{\bb{S}} * \dot{\bb{Q}} * \dot{\bb{T}}}$ is a forcing extension of $V^{\bb{C}^* * \dot{\bb{S}} * \dot{\bb{Q}} * \dot{\bb{T}}}$ by a $\lambda_1$-c.c. forcing, there is a club $D' \subseteq D$, $D' \in V^{\bb{C}^* * \dot{\bb{S}} * \dot{\bb{Q}} * \dot{\bb{T}}}$, such that $\Vdash_{\bb{C}^1_{n^*+1}/G(\bb{C}^*)}``\dot{\hat{T}} \cap D' = \emptyset"$. In $V^{\bb{C}^* * \dot{\bb{S}} * \dot{\bb{Q}}}$, let $\hat{\hat{T}} = \{\alpha < \lambda_1 \mid$for some $c \in \bb{C}^1_{n^*+1}/G(\bb{C}^*)$, $c \Vdash ``\alpha \in \dot{\hat{T}}"\}$. Then, in $V^{\bb{C}^* * \dot{\bb{S}} * \dot{\bb{Q}} * \dot{\bb{T}}}$, $\Vdash_{\bb{C}^1_{n^*+1}/G(\bb{C}^*)}``\dot{\hat{T}} \subseteq \hat{\hat{T}}"$, and $\hat{\hat{T}} \cap D' = \emptyset$. Thus, $\hat{\hat{T}}$ is a subset of $S^{\lambda_1}_{\geq \lambda_0}$ that is non-stationary in $V^{\bb{C}^* * \dot{\bb{S}} * \dot{\bb{Q}} * \dot{\bb{T}}}$ and is thus already non-stationary in $V^{\bb{C}^* * \dot{\bb{S}} * \dot{\bb{Q}}}$. But $V^{\bb{C}^* * \dot{\bb{S}} * \dot{\bb{Q}}} \subseteq V^{\bb{P}_0 \times (\bb{P}_1 * \dot{\bb{S}} * \dot{\bb{Q}})}$ and, in $V^{\bb{P}_0 \times (\bb{P}_1 * \dot{\bb{S}} * \dot{\bb{Q}})}$, $\hat{\hat{T}} \supseteq T$, contradicting the fact that $T$ is stationary in $V^{\bb{P}_0 \times (\bb{P}_1 * \dot{\bb{S}} * \dot{\bb{Q}})}$.
	\end{proof}

	The rest of the proof is much as in Case 1. We provide some details for completeness. In $V^{\bb{C}^1_{n^*+1} * \dot{\bb{S}} * \dot{\bb{Q}} * \dot{\bb{T}}}$, $\kappa^1_{n^*}$ is supercompact and $\lambda_1 = (\kappa^1_{n^*})^{+\omega + 1}$. Fix a fine, normal measure $U^*$ on $\mathcal{P}_{\kappa^1_{n^*}}(\lambda_1)$ such that $U^*$ projects to $U^1_{n^*}$. Let $\theta$ be a sufficiently large, regular cardinal, and let $\mathfrak{A}$ be an expansion of $(H(\theta), \in)$ by a well-ordering and constants for all relevant sets. The next lemmas are as before.

	\begin{lemma}
		Let $E_0 = \{X \in \mathcal{P}_{\kappa^1_{n^*}}(\lambda_1) \mid $for some $\mathfrak{B} \prec \mathfrak{A}$, we have $X = \mathfrak{B} \cap \lambda_1$, $|X| = |\mathfrak{B}|$, and $X \in \bigcap_{A \in U^* \cap \mathfrak{B}} A\}$. Then $E_0 \in U^*$.
	\end{lemma}

	\begin{lemma}
		Let $E_1 = \{X \in \mathcal{P}_{\kappa^1_{n^*}}(\lambda_1) \mid X \cap \kappa^0_{n^*}$ is inaccessible, $\mathrm{otp}(X) = (X \cap \kappa^0_{n^*})^{+\omega + 1}$, and $\hat{T} \cap X$ is stationary in $\sup(X)\}$. Then $E_1 \in U^*$.
	\end{lemma}

	\begin{lemma} \label{absorbLemma2}
		Let $X \in E_0 \cap E_1$ such that $X \cap \kappa^1_{n^*} \in A^{p_1}_{n^*}$. Let $\mathfrak{B} \prec \mathfrak{A}$ witness that $X \in E_0$. Then there is $(p_0, p^*_1) \in \bb{P}^*$ such that:
		\begin{enumerate}
			\item{$(p_0,p^*_1) \leq (p_0,p_1)$.}
			\item{$\ell(p^*_1) = n^* + 1$.}
			\item{$\alpha^{p^*_1}_{n^*} = X \cap \kappa^0_{n^*}$.}
			\item{If $(p_0,p'_1) \in \bb{P}^* \cap \mathfrak{B}$ and $p'_1$ is a trivial extension of $p_1$, then $(p_0,p^*_1) \leq (p_0,p'_1)$.}
		\end{enumerate}
	\end{lemma}

	Let $X, \mathfrak{B}$, and $(p_0,p^*_1)$ be as given in Lemma \ref{absorbLemma2}. As in Case 1, we get that, for every $\gamma \in \hat{T} \cap X$, there is $(\dot{s}_\gamma, \dot{q}_\gamma)$ such that $(p_0, p^*_1)$ forces that $(\dot{s}_\gamma, \dot{q}_\gamma) \in G(\dot{\bb{S}} * \dot{\bb{Q}})$ and $(p_0, (p^*_1, \dot{s}_\gamma, \dot{q}_\gamma))$ forces that $\gamma \in \dot{T}^*$. Moreover, we may assume that, for every such $\gamma$, there is $\dot{t}_\gamma$ such that $(p_0, p^*_1)$ forces $(\dot{s}_\gamma, \dot{q}_\gamma, \dot{t}_\gamma)$ is in the dense $\lambda_1$-directed closed subset of $\dot{\bb{S}}*\dot{\bb{Q}}*\dot{\bb{T}}$ and in $G(\dot{\bb{S}} * \dot{\bb{Q}} * \dot{\bb{T}})$. We can thus find names $\dot{s}'$ and $\dot{q}'$ such that, for all $\gamma \in \hat{T} \cap X$, $(p_0, (p^*_1, \dot{s}', \dot{q}')) \leq (p_0, (p^*_1, \dot{s}_\gamma, \dot{q}_\gamma))$. Since, when forcing with $\bb{P}^*$ below $(p_0, q_1)$, no cardinals between $X \cap \kappa^1_{n^*}$ and $(X \cap \kappa^1_{n^*})^{+\omega+1}$ are collapsed, another application of Claim \ref{unbddSubset} yields that $(p_0, (q_1, \dot{s}', \dot{q}'))$ forces that $\dot{T}$ reflects at $\sup(X)$.

	It remains to show that $AP_{\kappa^1_\omega}$ fails in $V^{\bb{P}_0 \times (\bb{P}_1 * \dot{\bb{S}} * \dot{\bb{Q}})}$. We will use an equivalent alternative formulation of approachability, due to Shelah.

	\begin{definition}
		Suppose $\kappa$ is a singular cardinal of countable cofinality, and let $d:[\kappa^+]^2 \rightarrow \omega$.
		\begin{enumerate}
			\item{$d$ is \emph{normal} if, for all $\beta < \kappa^+$ and all $n < \omega$, $|\{\alpha < \beta \mid d(\alpha,\beta) \leq n \}| < \kappa$.}
			\item{$d$ is \emph{subadditive} if, for all $\alpha < \beta < \gamma < \kappa^+$, $d(\alpha, \gamma) \leq \max(d(\alpha,\beta),d(\beta,\gamma))$.}
			\item{$S_0(d)$ is the set of $\gamma < \kappa^+$ such that, for some unbounded sets $A,B \subseteq \gamma$, for every $\beta \in B$, there is $n_\beta < \omega$ such that, for all $\alpha \in A \cap \beta$, $d(\alpha, \beta) \leq n_\beta$.}
		\end{enumerate}
	\end{definition}

	\begin{lemma}(Shelah)
		Suppose $\kappa$ is a singular, strong limit cardinal of countable cofinality.
		\begin{enumerate}
			\item{There is a normal, subadditive function $d:[\kappa^+]^2 \rightarrow \omega$.}
			\item{If $d$, $d' : [\kappa^+]^2 \rightarrow \omega$ are two normal, subadditive functions, then $S_0(d) \triangle S_0(d')$ is non-stationary.}
			\item{$AP_\kappa$ is equivalent to the existence of a normal, subadditive $d:[\kappa^+]^2 \rightarrow \omega$ such that $S_0(d)$ contains a club.}
		\end{enumerate}
	\end{lemma}

	In $V$, fix a normal, subadditive $d:[\lambda_1] \rightarrow \omega$. Note that $d$ remains normal and subadditive in $V^{\bb{P}_0 \times (\bb{P}_1 * \dot{\bb{S}} * \dot{\bb{Q}})}$. Let $(p_0, (p_1, \dot{s}, \dot{q})) \in \bb{P}_0 \times (\bb{P}_1 * \dot{\bb{S}} * \dot{\bb{Q}})$. Move to $V^{\bb{C}^1_{\ell(p_1)+1} * \dot{\bb{S}} * \dot{\bb{Q}} * \dot{\bb{T}}}$, requiring that, letting $c$ be the $C$-part of $p_1$, $(c, \dot{s}, \dot{q}) \in G(\bb{C}^1_{\ell(p_1)+1} * \dot{\bb{S}} * \dot{\bb{Q}})$. $\kappa^1_{\ell(p_1)}$ remains supercompact in $V^{\bb{C}^1_{\ell(p_1)+1} * \dot{\bb{S}} * \dot{\bb{Q}} * \dot{\bb{T}}}$, and a standard application of supercompactness yields that, if $A = \{\alpha < \kappa^1_{\ell(p_1)} \mid S^{\lambda_1}_{\alpha^{+\omega+1}} \setminus S_0(d)$ is stationary$\}$, then $A \in U^1_{\ell(p)}$. Note that, since this is true in $V^{\bb{C}^1_{\ell(p_1)+1} * \dot{\bb{S}} * \dot{\bb{Q}} * \dot{\bb{T}}}$, it must be true in $V$ as well. Moreover, by previous arguments, for any club $D$ in $\lambda_1$ in $V^{\bb{P}_0 \times (\bb{P}_1 * \dot{\bb{S}} * \dot{\bb{Q}})}$, there must be a club $C \subseteq D$ in $V^{\bb{C}^1_{\ell(p_1)+1} * \dot{\bb{S}} * \dot{\bb{Q}} * \dot{\bb{T}}}$. Putting this together, working in $V$, there are $U^1_{\ell(p_1)}$-many $\alpha < \lambda_1$ such that $(S^{\lambda_1}_{\alpha^{+\omega + 1}} \setminus S_0(d))^V$ is stationary in $V^{\bb{P}_0 \times (\bb{P}_1 * \dot{\bb{S}} * \dot{\bb{Q}})}$. Find $\alpha \in A \cap A^{p_1}_{\ell(p_1)}$, and find an extension $q_1 \leq p_1$ such that $\ell(q_1) = \ell(p_1) + 1$ and $\alpha^{q_1}_{\ell(p_1)} = \alpha$. It suffices to show that, forcing below $(p_0, (q_1, \dot{s}, \dot{q}))$, $(S^{\lambda_1}_{\alpha^{+\omega+1}} \setminus S_0(d))^{V^{\bb{P}_0 \times (\bb{P}_1 * \dot{\bb{S}} * \dot{\bb{Q}})}} = (S^{\lambda_1}_{\alpha^{+\omega+1}} \setminus S_0(d))^V$. To this end, let $\beta \in (S^{\lambda_1}_{\alpha^{+\omega + 1}} \cap S_0(d))^{V^{\bb{P}_0 \times (\bb{P}_1 * \dot{\bb{S}} * \dot{\bb{Q}})}}$. Let $A,B$ be unbounded in $\beta$ witnessing $\beta \in S_0(d)$. Since all cardinals in the interval $(\alpha, \alpha^{+\omega+2})$ are preserved by the forcing, Claim \ref{unbddSubset} yields unbounded $A' \subseteq A$ and $B' \subseteq B$ such that $A',B' \in V$. But then $A'$, $B'$ witness that $\beta \in S_0(d)$ in $V$. Thus, $(S^{\lambda_1}_{\alpha^{+\omega+1}} \setminus S_0(d))^{V^{\bb{P}_0 \times (\bb{P}_1 * \dot{\bb{S}} * \dot{\bb{Q}})}} = (S^{\lambda_1}_{\alpha^{+\omega+1}} \setminus S_0(d))^V$, so $AP_{\kappa^1_\omega}$ fails in $V^{\bb{P}_0 \times (\bb{P}_1 * \dot{\bb{S}} * \dot{\bb{Q}})}$.
\end{proof}

\section{Questions}

Many questions remain about the possible patterns of stationary reflection at the successor of a singular cardinal. We ask only a few of them here.

\begin{question}
	Is it consistent that $\mathrm{bRefl}(\aleph_{\omega^2 + 1})$ and $\neg AP_{\aleph_{\omega^2}}$ hold simultaneously?
\end{question}

\begin{question}
	Is it consistent that $\mathrm{Refl}(\aleph_{\omega^2+1})$ holds and, for every stationary $S \subseteq \aleph_{\omega^2 + 1}$, there is a stationary $T \subseteq S$ that does not reflect at arbitrarily high cofinalities?
\end{question}

\begin{question}
	Is it consistent that $\mathrm{Refl}(\aleph_{\omega \cdot 2 + 1})$ holds and there is a stationary $S \subseteq S^{\aleph_{\omega \cdot 2 + 1}}_\omega$ that does not reflect at any ordinal in $S^{\aleph_{\omega \cdot 2 + 1}}_{<\aleph_\omega}$?
\end{question}

\bibliography{reflection2ref}
\bibliographystyle{plain}

\end{document}